\documentclass[11pt,english,reqno]{amsart}

\usepackage{bm}
\usepackage{color}
\usepackage{babel}
\usepackage{verbatim}
\usepackage{amsthm}
\usepackage{amstext}
\usepackage{amssymb}
\usepackage{amsbsy}
\usepackage{ifpdf}
\usepackage[nobysame,abbrev,alphabetic]{amsrefs}
\usepackage{enumerate}
\usepackage{amsmath}
\usepackage{amscd}
\usepackage{amsfonts}
\usepackage{graphicx}
\usepackage[all]{xy}
\usepackage{verbatim}
\usepackage{gensymb}
\usepackage{mathrsfs}
\usepackage[colorlinks]{hyperref}
\hypersetup{
linkcolor=blue,          
citecolor=green,        
}
\newtheorem{theorem}{Theorem}[section]

\newtheorem{lemma}[theorem]{Lemma}
\newtheorem{corollary}[theorem]{Corollary}
\theoremstyle{definition}\newtheorem{definition}[theorem]{Definition}

\newtheorem{example}[theorem]{Example}
\newtheorem{conjecture}[theorem]{Conjecture}

\theoremstyle{theorem}

\newtheorem{problem}[theorem]{Problem}
\theoremstyle{definition}

\theoremstyle{definition}
\theoremstyle{definition}\newtheorem{remark}[theorem]{Remark}
\theoremstyle{definition}\newtheorem*{acknowledgments}{Acknowledgments}
\newtheorem{assumption}[theorem]{Assumption}
\newcommand{\al}{\alpha}
\newcommand{\be}{\beta}

\newcommand{\Del}{\Delta}
\newcommand{\lam}{\lambda}

\newcommand{\eps}{\epsilon}

\newcommand{\sig}{\sigma}

\newcommand{\om}{\omega}
\newcommand{\Om}{\Omega}
\newcommand{\vphi}{\varphi}

\newcommand{\cF}{\mathcal{F}}

\newcommand{\cO}{\mathcal{O}}

\newcommand{\bC}{\mathbb{C}}

\newcommand{\bP}{\mathbb{P}}
\newcommand{\bR}{\mathbb{R}}
\newcommand{\bZ}{\mathbb{Z}}
\newcommand{\bQ}{\mathbb{Q}}

\newcommand{\bK}{\mathbb{K}}
\newcommand{\bN}{\mathbb{N}}
\newcommand{\bH}{\mathbb{H}}

\newcommand{\SL}{\operatorname{SL}}
\newcommand{\SO}{\operatorname{SO}}

\newcommand{\defi}{\overset{\on{def}}{=}}

\newcommand\norm[1]{||#1||}

\newcommand\set[1]{\left\{#1\right\}}
\newcommand\pa[1]{\left(#1\right)}

\newcommand\av[1]{|#1|}
\newcommand\on[1]{\operatorname{#1}}
\newcommand\diag[1]{\operatorname{diag}\left(#1\right)}

\newcommand\mb[1]{\mathbf{#1}}

\newcommand\smallmat[1]{\pa{\begin{smallmatrix}#1\end{smallmatrix}}}

\newcommand{\limfi}[2]{{\displaystyle \lim_{#1\to#2}}}

\newcommand{\onto}{\xymatrix{\ar@{>>}[r]&}}

\newcommand{\odadd}[1]{{\color{blue}{\tiny [OD]} #1}}

\begin{document}
\title{Shapes of unit lattices and escape of mass}
\author[Ofir David]{Ofir David}
\author[Uri Shapira]{Uri Shapira}

\begin{abstract}
We study the collection of points on the modular surface obtained from the logarithm embeddings of the 
groups of units in totally real cubic number fields. We conjecture that this set is dense and  
show that its closure contains countably many explicit curves and give a 
strategy to show that it has non-empty interior. The results are obtained by constructing explicit families of orders (generalizing the 
so called ``simplest cubic fields") and calculating their groups of units. We also address the question of escape of mass for the compact 
orbits of the diagonal group associated to these orders. 
\end{abstract}
\address{Department of Mathematics\\
Technion \\
Haifa \\
Israel }
\email{ofirdav@tx.technion.ac.il}

\address{Department of Mathematics\\
Technion \\
Haifa \\
Israel }
\email{ushapira@tx.technion.ac.il}
%
%
\maketitle

\section{Introduction}
This paper originates from an attempt to understand concrete examples of sequences of compact
orbits for the diagonal group $A<\SL_3(\bR)$ on the space of lattices $X \defi \SL_3(\bR)/\SL_3(\bZ)$.
We investigate two seemingly unrelated questions one can ask about such orbits.
\subsection{Shapes of unit lattices}
We begin by explaining
the first question which we find most interesting and wish to promote its study.
It is well known (see \S\ref{Preliminaries}) that given an order $\cO$ in a totally real cubic number field, one can construct out of it
a lattice with a compact $A$-orbit whose geometric shape is governed by the shape of the group of units 
$\cO^\times$.

More precisely, 
let $\set{\sig_i}_1^3$ denote the embeddings of the field into the reals. 
Dirichlet's unit theorem states that if we denote for
$\om\in \cO^\times$,
$\psi(\om)\defi (\log\av{\sig_1(\om)},\log\av{\sig_2(\om)},\log\av{\sig_3(\om)})$, then $\psi$ maps $\cO^\times$ to a
lattice in the plane
$\bR^3_0\defi\set{\mb{t}\in\bR^3:\sum_1^3t_i=0}$. In turn, we may define the \textit{shape} $\Del_{\cO^\times}$
of the unit lattice $\cO^\times$
to be the corresponding point on the modular curve $\SL_2(\bZ)\backslash \bH$. This correspondence is defined as follows: one
chooses a similarity map to identify $\bR^3_0$ and $\bR^2$ which maps $\psi(\cO^\times)$ to a unimodular lattice in $\bR^2$; i.e.\
to a point in $\SL_2(\bR)/\SL_2(\bZ)$. Since the similarity is only well defined up to rotation, we obtain a well defined point in
$\SO_2(\bR)\backslash\SL_2(\bR)/\SL_2(\bZ)\simeq \SL_2(\bZ)\backslash \bH$. We set
$$\Om\defi \set{\Del_{\cO^\times}\in \SL_2(\bZ)\backslash \bH:\cO\textrm{ is an order in a totally real cubic number field}}.$$
We wish to promote the following conjectures.
\begin{conjecture}\label{conj:main}
\begin{enumerate}
\item\label{conj:nc} The closure $\overline{\Om}$ in the modular surface is non-compact.
\item\label{conj:int} The closure $\overline{\Om}$ in the modular surface has non-empty interior.
\item\label{conj:dense} The set  $\Om$ is dense in the modular surface.
\end{enumerate}
\end{conjecture}
Despite the fact that the above conjectures are natural, as far as we know there is virtually nothing in the literature about them.
In personal communication with Andre Reznikov we learned that questions which are similar in spirit to the above 
were also suggested by Margulis and Gromov
and that numerical experiments seem to support Conjecure~\ref{conj:main}. 
We provide modest progress towards Conjecture~\ref{conj:main}(\ref{conj:int}) and prove that
$\overline{\Om}$ contains countably many  explicit curves illustrated in Figure~\ref{fig:Shapes-of-units}. 
For more details see Theorem~\ref{thm:curves}.
\begin{figure}[h]
\begin{centering}
\includegraphics[scale=0.75]{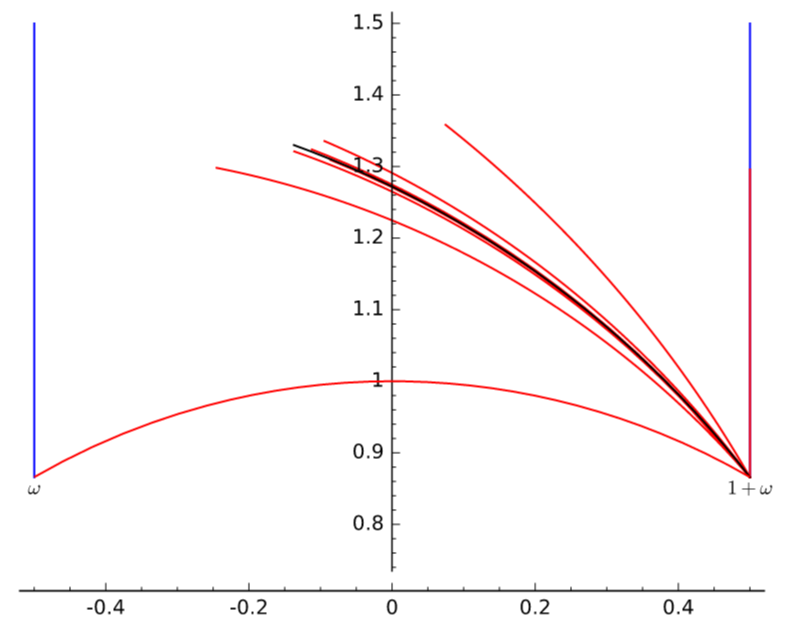}
\par\end{centering}

\caption{Appearing in red are continuous curves (with accumulations) of cluster points of shapes of unit lattices of
rings of the form $\protect\bZ\left[\theta\right]$
where $\theta$ is a unit, which we construct. The black curve is (one of the countably many) limits of the red curves.}\label{fig:Shapes-of-units}
\end{figure}

In Problem~\ref{prob:cong} we give a strategy of how to reduce Conjecture~\ref{conj:main}(\ref{conj:int}) to a certain problem in finding enough solutions to some congruence conditions.

\subsection{Escape of mass}
We now describe the second question we study. One of the major open questions in homogeneous dynamics is to understand the space of $A$-invariant
and ergodic probability measures on $X$. Conjecturally, this space is composed of periodic measures only. Here, a probability
measure is called \textit{periodic} if it is
$L$-invariant and supported on a single orbit $Lx\subset X$, where $L<\SL_3(\bR)$ is a closed subgroup. In such a case we denote this
measure by $\mu_{Lx}$ and say that the orbit $Lx$ is a \textit{periodic orbit}.
In fact, in the above example, due to scarcity of closed groups $A<L<\SL_3(\bR)$, the $A$-invariant and ergodic periodic measures in this space are $\mu_{X}$ -- the
unique $\SL_3(\bR)$-invariant probability measure on $X$ -- and the ones corresponding to periodic
$A$-orbits (which is a synonym for compact $A$-orbits). Apart from describing what are the $A$-invariant and ergodic probability measures on $X$, it is desirable to understand
the topology of this space. In particular, what can be said about the weak* accumulation points of sequences $\mu_{Ax_n}$
of periodic $A$-invariant measures supported on compact $A$-orbits. The question that we study for a sequence $\mu_{Ax_n}$
is that of \textit{partial} or \textit{full escape of mass}. We say that a sequence $\mu_{Ax_n}$ exhibits $c$-\textit{escape of mass} for  $0<c\le 1$ if any weak$^*$ accumulation point $\mu$ of $\mu_{Ax_n}$ satisfies $\mu(X)\le1-c$. We say that it exhibits \textit{full escape of mass} if it exhibits $1$-escape of mass, i.e.\ if $\mu_{Ax_n}$ converges
to the zero measure. 

The reason for the number theoretic interest in periodic $A$-orbits is that they correspond to full modules in totally real
cubic number fields as will be discussed later.
Our aim in this direction is to review and construct particular examples of such full modules and establish partial or full escape of mass of the corresponding
orbits. In practice, what we do is exhibit a family of cubic polynomials $\set{f_i(x):i\in I}\subset \bZ[x]$ which depend on some parameter $i\in I$ 
(such that $f_i(x)$ is irreducible and totally real), and
discuss the periodic $A$-orbit corresponding to the order $\bZ[\theta_i]$, where $\theta_i$ is a root of $f_i(x)$, as the parameter varies.
One might expect that
if the polynomials are chosen carefully, then conclusions regarding the orbits could be derived.

\subsection{Structure of the paper and results} 
In \S\ref{Preliminaries}
we give the general notation and correspondence between full modules
of general orders in number fields and periodic $A$-orbits in $X$.
In particular, we will give a condition on
the relation between the discriminant and the unit group that will
be sufficient to produce escape of mass.

As stated before, we are interested in lattices arising from rings of the form $\bZ[\theta]$ where $\theta$ is the root of some monic irreducible polynomial $f(x)$. In particular, we will be interested in the case where the units of $\bZ[\theta]$ are generated by elements of the form $a\theta-b,c\theta-d$ (and $-1$).

We start this investigation in \S\ref{Construction-of-cubic}, and show in Lemma~\ref{lem:unit_condition} that a necessary condition for $a\theta-b, c\theta-d$ to be units in $\bZ[\theta]$ is that $a^3 f(\frac{b}{a})=\pm 1$ and similarly $c^3 f(\frac{d}{c})=\pm 1$, which is a solution to two integral equations in the coefficients of $f$.

In  \S\ref{zero_condition}, \S\ref{nonzero_condition}  we will show how to construct monic cubic polynomials 
$f_{a,b,c,d,t}\left(x\right)$ (all parameters being integers) which satisfy these conditions. Moreover, we will show that there are infinitely  many such polynomials (parametrized by $t$) whenever $a,b,c,d$ satisfy
a simple congruence conditions, and we will give some examples for such $a,b,c,d$. 

In \S\ref{Full-escape-of}, we fix the parameters $a,b,c,d$ and
take $|t|\to\infty$. In this case, for $|t|$ big enough, the polynomials $f_{a,b,c,d,t}(x)$ will be irreducible and not only will the elements $a\theta-b,c\theta-d$ be integral units, they will actually form a set of fundamental units (i.e. they generate the unit group
together with $-1$). More precisely, we have the following result which is a direct consequence of Theorem~\ref{approx_roots} and Theorem~\ref{thm:fixed_parameters}.

\begin{theorem}\label{thm:full_escape_ex}
Fix $a,b,c,d\in\bZ$ such that $\frac{b}{a}\neq\frac{d}{c}, \;a\neq \pm c$ and there
exists a monic polynomial $h(x)\in \bZ[x]$ satisfying $a^{3}h\left(\frac{b}{a}\right)=\epsilon_{1},\;c^{3}h\left(\frac{d}{c}\right)=\epsilon_{2}$, where $\eps_i=\pm1$.
We denote $h_{t}(x)=h(x)+tg(x),$ $g(x)=\left(ax-b\right)\left(cx-d\right)$
where $t\in\bZ$. Then the following holds
\begin{enumerate}
\item For all $\left|t\right|$ big enough the polynomial $h_{t}\left(x\right)$
is totally real and irreducible.
\item Setting $\theta_t$ to be a root for $h_t$, for all $|t|$ big enough the unit group of $\bZ[\theta_t]$ is generated by $\{a\theta_t-b, c\theta_t-d,-1\}$.
\item As $|t|\to \infty$ the shape of the unit lattice converges to the regular triangles lattice.
\item As $|t| \to \infty$, the orbits corresponding to the orders $\bZ[\theta_t]$ exhibit full escape of mass.
\end{enumerate}
\end{theorem}
We remark two things: 
(i) What stands behind Theorem~\ref{thm:full_escape_ex} is that the family of polynomials $h_t$ is controled by a degree 2 polynomisl $g$, and as $|t|$ increases, we can approximate the roots of $h_t$ using the roots of $g$.
(ii) As mentioned above, the existence of a polynomial $h$ used to jumpstart Theorem~\ref{thm:full_escape_ex} is guaranteed by a simple congruence 
condition on the parameters $a,b,c,d$.  

The phenomena described in Theorem~\ref{thm:full_escape_ex} are similar in nature to what happens in~\cite{Shapira} and in particular, the shape 
of the regular triangle lattice is the only possible limit shape. It turns out that in order to create new limit shapes one needs to vary the parameters 
$a,b,c,d$ with $t$.
This approach is implemented in \S\ref{one_unit}. In fact, to simplify matters
we concentrate on the case where $c=1,d=0$ 
(i.e.\ that $\theta$ is an integral unit), and take $a:=a_t,b:=b_t$ to increase
to infinity as $|t|\to\infty$. As mentioned earlier, in order to construct the relevant polynomial for such $a_t,b_t$ they need to satisfy a simple congruence condition
which we now define.
\begin{definition}
We say that a pair of integers $(a,b)$ is a \textit{mutually cubic root pair} if
$a^{3}\equiv_{b} 1$ and $b^{3}\equiv_{a} 1$. 
A sequence $(a_t,b_t)$ is called a \textit{mutually cubic root sequence} if $(a_t,b_t)$ is a mutually cubic root pair for any $t\in \bN$ or $\bZ$.
\end{definition}
Given a mutually cubic root sequence, we are able to construct a family of orbits which exhibit (at least) partial escape of mass. Furthermore, we will also compute the shapes of the unit lattices and their limit
as $|t| \to \infty$. Unlike the case with $a,b,c,d$ fixed, here the limit shapes will not necessarily be the regular triangles lattice.
\begin{theorem}\label{thm:curves}
Let $(a_t,b_t)$ be a mutually cubic root sequence and suppose that the limits $\tilde{a}=\limfi t{\infty}\frac{\log\left|a_t\right|}{\log\left|t\right|}$ and $\tilde{b}=\limfi t{\infty}\frac{\log\left|b_t\right|}{\log\left|t\right|}$ exist and satisfy $0\leq \tilde{a}\leq \tilde{b}$. Then $\overline{\Omega}\subseteq \SL_2(\bZ) \backslash \mathbb{H}$ contains the image of the curve 
$$\gamma(r)=\frac{1+2r\tilde{a}+\left(1+r\tilde{b}+2r\tilde{a}\right)\omega}
{1+r\tilde{a}+\left(r\tilde{a}-r\tilde{b}\right)\omega}, \; \; r\in \left[0,\min(\frac{1}{3\tilde{a}},\frac{1}{\tilde{b}})\right],$$
where $\om= e^{\frac{2\pi i}{3}}$.
\end{theorem}
We note that the ratios 
$\frac{\tilde{a}}{\tilde{b}}=\limfi t{\infty}\frac{\log\left|a_t\right|}{\log\left|b_t\right|}$ (thought of as points in $P^1(\bR)$), and the curves 
are in 1-1 correspondence. 
In \S\ref{many_examples} we will show how to produce infinitely
many examples of mutually cubic root sequences $(a_t,b_t)$, which in turn produce countably many
distinct limits of the form $\limfi t{\infty}\frac{\log\left|a_t\right|}{\log\left|b_t\right|}$. 
The shape of unit lattices produced by these orders can be seen in Figure~\ref{fig:Shapes-of-units}.
As a consequence to the previous theorem, it is straightforward to see that a positive solution to the
following problem will imply Conjecture~\ref{conj:main}(\ref{conj:int}).
\begin{problem}\label{prob:cong}
Let $\Lambda\subset \bP(\bR)$ be the set of all the possible ratios $\frac{\tilde{a}}{\tilde{b}}$, where 
$\tilde{a}=\lim\frac{\log{a_t}}{\log t}$, $\tilde{b}=\lim\frac{\log{b_t}}{\log t}$ (not both zero), and $(a_t,b_t)$ 
is a mutually cubic root sequence. 
Is the interior of $\Lambda$
nonempty?
\end{problem}
We remark that in Corollary~\ref{cor:accpts} we show that $\Lambda$
 has infinitely many accumulation points.
 
 \subsection{Comparison with earlier results} 
This work is a succession of the discussion in~\cite{Shapira} in which the second named author addressed the above questions in regards to 
certain sequences of compact $A$-orbits (in any dimension). In that discussion all sequences of compact $A$-orbits exhibited full escape of mass
but more interestingly, the shapes 
of unit lattices there converged to a fixed shape which in dimension 3 is the shape of the regular triangle lattice which corresponds to the point
$\om=\exp(\frac{2\pi i}{3})$ on the corner of the fundamental domain in Figure~\ref{fig:Shapes-of-units}. Interestingly, in the present work when
we produce examples of orders for which the shapes 
of unit lattices converge to a shape not equal to $\om$ we only manage to establish partial escape of mass.

The problem of finding generators for the group of units $\cO^\times$ is classical. Explicit examples of computations  
may be found in \cite{minemura_totally_1998,thomas_fundamental_1979,grundman_systems_1995,minemura_totally_2004}, 
and most notably in the spirit of the current discussion, in~\cite{cusick_regulator_1991} where it is shown that the curve at the bottom of the fundamental domain in Figure~\ref{fig:Shapes-of-units},
$e^{2\pi i\theta}, \theta\in[\pi/3,2\pi/3]$ is contained in $\overline{\Om}$. 

\section{\label{Preliminaries}Preliminaries}
We now set up the number theoretic notation and terminology needed for our discussion.
For the general background from number theory see
for example \cite{algebraic_1999}.

Let $\bK/\bQ$ be a totally real number field of degree $n$.
A \emph{full module} $M$ in $\bK$ is an abelian subgroup $M=sp_{\bZ}\left\{ \alpha_{1},...,\alpha_{n}\right\} \leq\bK$ such that $\bQ M=\bK$. An \emph{order}  in $\bK$ is a full module which is also a unital ring. We denote by $\cO_\bK$ the ring of integers
which is the unique maximal order in $\bK$.
Let
$\sigma_{1},...,\sigma_{n}:\bK\to\bR$ be the $n$ distinct real embeddings of $\bK$.
The homomorphism $\varphi:\bK\to\bR^{n}$ defined by $\varphi(\alpha):=\left(\sigma_{i}(\alpha)\right)_{1}^{n}$
is an embedding which sends any full-module $M<\bK$
to a lattice in $\bR^{n}$. The \textit{discriminant} $D_M$ of $M$ is defined as the square of the covolume 
of $\varphi(M)$.  We denote by $D_{\bK}$
 the discriminant of of the ring of integers $\cO_\bK$.
Given a full module $M$ we define the \textit{associated order} of $M$ to be 
$\cO_{M}:=\left\{ \alpha\in\bK\;\mid\;\alpha M\subseteq M\right\}$, and denote by $\cO_M^\times$ the group of units of $\cO_M$. 
Note that $M$ is itself an order if and only if $M=\cO_M$.
Let $\psi:\bK\to\bR^n$ be defined by $\psi(\al)=(\log\av{\sig_i(\al)})_{1}^n$. Since the norm of a unit is $\pm 1$, $\psi(\cO_M^\times)\subset
\bR^n_0:=\set{x\in\bR^n:\sum_1^nx_i=0}$. Dirichlet's unit theorem says that $\psi(\cO_M^\times)$ is a lattice in $\bR^n_0$.
A collection $\set{\al_j}_1^{n-1}\subset \cO_M^\times$ is a \textit{fundamental set of units} if $\set{\psi(\al_j)}_1^{n-1}$ forms a
basis for $\psi(\cO_M^\times).$ The \textrm{Regulator} $R_M$
of $M$ is defined as the covolume of  the projection of $\psi(\cO_M^\times)$ into any copy of $\bR^{n-1}$ spanned by the axis in $\bR^n$. 
Equivalently, if 
$\left\{ \alpha_{j}\right\}_1^{n-1} \leq\cO_{M}^{\times}$
is a fundamental set of units
then $R_M$ is the
determinant of any $\left(n-1\right)\times\left(n-1\right)$ submatrix
of of the matrix $\left(\log\left|\sigma_{i}\left(\alpha_{j}\right)\right|\right)$
where $1\leq i\leq n$ and $1\leq j\leq n-1$.  If $\left\{ \alpha_{j}\right\} _{1}^{n-1}$ is
just a set of independent units, we shall call this determinant the
\emph{relative regulator}.

We now restrict our attention to orders in totally real cubic fields.
The following theorem and its corollary will give us the tool to prove that a pair of units is a fundamental pair.
This was used also in \cite{cusick_regulator_1991}.  
\begin{theorem}
[Cusick \cite{cusick_lower_1984}]\label{thm:cusik} 
For an order in a totally real cubic number field of discriminant $D$ and
regulator $R$, one has $\frac{R}{\log^{2}\left(\frac{D}{4}\right)}\geq\frac{1}{16}$.
In particular, for any sequence of such orders
with discriminants and regulators $D_{i},R_{i}$ respectively, we
have ${\displaystyle \liminf_{i\to\infty}}\frac{R_{i}}{\log^{2}\left(D_{i}\right)}\geq\frac{1}{16}$.
\end{theorem}
We remark that the formulation of this result in \cite{cusick_lower_1984} is for maximal orders but 
that the proof works verbatim for a general order.
\begin{corollary}
\label{fundamental_units}Let $\bK$ be a totally real cubic field,
and let $M\leq\bK$ be an order with discriminant $D$ and regulator
$R$. If $\left\{ \alpha_{1},\alpha_{2}\right\} \leq M^{\times}$
is an independent set of units with relative regulator $R'$ such
that $\frac{R'}{\log^{2}\left(\frac{D}{4}\right)}<\frac{1}{8}$, then
they must be a fundamental set. \end{corollary}
\begin{proof}
If $\left\{ \alpha_{1},\alpha_{2}\right\} $ is not a fundamental
set, then the regulator of $M$ would satisfy $R\leq\frac{1}{2}R'$
and therefore $\frac{R}{\log^{2}\left(\frac{D}{4}\right)}<\frac{1}{16}$
which is a contradiction to the previous theorem.
\end{proof}
We briefly describe the relation between full modules and compact $A$-orbits in the space of lattices.
The space of unimodular lattices  is identified as usual with $X:= \SL_n(\bR)/\SL_n(\bZ)$ and we denote by $A\leq \SL_n(\bR)$ the subgroup of positive diagonal matrices.
Given a full module $M$ in a totally real degree $n$ number field $\bK$ with embeddings $\set{\sig_i}_1^n$, we denote $L_M:=D_M^{-\frac{1}{2n}}\vphi(M)\in X$.
The compactness of the orbit $AL_M$ is a consequence of Dirichlet's theorem as we now explain. 
This compactness is equivalent to the statement that 
$\on{stab}_{\bR^n_0}(L_M):=\set{x\in\bR^n_0: \exp(x)L_M=L_M}$ is a lattice in 
$\bR^n_0$, where here $\exp:\bR^n_0\to A$ is given by $\exp(x):=\diag{e^{x_1},\dots,e^{x_n}}$.
It is straightforward that for $\alpha\in \bK^\times$, $L_{\alpha M}=a_\alpha L_M$ where $a_\alpha:=\diag{\sig_1(\al),\dots,\sig_n(\al)}$ on the diagonal. Therefore, if $\alpha\in \cO_M^\times$ then $a_\alpha L_M=L_M$ and $\det a_\al=\pm 1$ (because $\alpha$ has norm $\pm 1$). If all the 
$\sig_i(\al)$'s are positive then $\psi(\al)\in\on{stab}_{\bR^n_0}(L_M)$. In fact, the converse is also true (see~\cite{ELMV_periodic_torus, LW, McMullenMinkowski, ShapiraWeiss}),that is, if we set
$\cO_M^{\times,+}:=\set{\al\in \cO_M^\times:\forall i, \sig_i(\al)>0}$, then 
\begin{equation*}\label{eq:stab}
\psi(\cO_M^{\times,+}) =  \on{stab}_{\bR^n_0}(L_M).
\end{equation*} 
Now since $\psi(\cO_M^{\times,+})$ is a finite index subgroup of $\psi(\cO_M^\times)$, and the latter is a lattice
in $\bR^n_0$ by Dirichlet's theorem, we conclude that 
$\psi(\cO_M^{\times,+})$ is a lattice as well.

\begin{remark}\label{rem:shape of orbit}
We note two things. First, it is a classical fact (that we will not use), that all compact $A$-orbits are of the form $AL_M$ for some full module $M$ as above
(see any of  ~\cite{ELMV_periodic_torus, LW, McMullenMinkowski, ShapiraWeiss}). 
Second, although strictly speaking, the shape of the orbit $AL_M$ should be defined to be the (equivalence class up to similarity of) lattice $\psi(\cO_M^{\times,+})$,
it is much more natural from the number theoretic point of view to work with the lattice $\psi(\cO_M^\times)$. Although by referring to the equivalence class of the latter lattice as the \textit{shape} of the orbit (and not as the shape of the unit lattice) we are a bit misleading, we find this slight abuse harmless.
\end{remark}

We turn now to present 
the necessary tools to establish the escape of mass in our results. 
For a more thorough
discussion the reader is referred to~\cite{Shapira}.
\begin{definition}
Let $L\in X$ be a unimodular lattice.
\begin{enumerate}
\item We define the height of $L$ to be
\[
ht(L)=\left(\min\left\{ \norm v\;\mid\;0\neq v\in L\right\} \right)^{-1}=\max\left\{ \norm v^{-1}\;\mid\;0\neq v\in L\right\} .
\]

\item For $H\geq0$ we define $X^{\leq H}$ (resp. $<,\geq,>$) by
\[
X^{\leq H}=\left\{ L\in X\;\mid\;ht(L)\leq H\right\} .
\]

\end{enumerate}
\end{definition}
The sets $X^{\leq H}$ are compact and $X=\bigcup_{H}X^{\leq H}$. 
The statement that a sequence of periodic $A$-orbits $Ax_k$ exhibits $c$-escape of mass for $0<c\le 1$  is equivalent to the statement that for any $H>0, \eps>0$ and any $k$ large enough $\mu_{Ax_k}(X^{\ge H})\ge c-\eps$. 

The minor difference between $\psi(\cO_M^{\times,+})$ and $\psi(\cO_M^\times)$ does not play any role in the discussion of escape of mass because of the following.
\begin{lemma}
Let $M$ be a full module in a totally real number field as above. 
The height map $h:\bR^n_0/\psi(\cO_M^{\times,+})\to \bR$ given by $h(x):=ht(\exp(x)L_M)$ factors through
$\bR^n_0/\psi(\cO_M^\times)$.
\end{lemma}
\begin{proof}
if $x,y\in\bR^n_0$ are such that $x-y\in\psi(\cO_M^\times)$ then there exists $\al\in\cO_M^\times$ and a diagonal $\pm 1$ matrix
$J_\al$ such that $\exp(x-y) =J_\al a_\al$. Since $J_\al$ acts as an isometry on $\bR^n$ and $a_\al L_M=L_M$ we get that 
\begin{align*}
h(x)&=ht(\exp(x)L_M)=ht(\exp(y)\exp(x-y)L_M)\\
&=ht(\exp(y)J_\al a_\al L_M)
= ht(J_\al\exp(y)L_M)=ht(\exp(y)L_M)=h(y).
\end{align*}
\end{proof} 
\begin{corollary}\label{cor:escape}
Let $AL_M$ be a compact $A$-orbit as above, let $F$ be a fundamental domain for $\psi(\cO_M^\times)$ in 
$\bR^n_0$, let $\lam$ denote the Lebesgue measure on $\bR^n_0$ and let $\mu_{AL_M}$ be 
the periodic $A$-invariant probability measure on the orbit $AL_M$. Let $h:\bR^n_0\to\bR$ be the height function $h(x):=ht(\exp(x)L_M)$.
Then, for any $H>0$ we have $\mu_{AL_M}(X^{> H}) = \frac{1}{\lam(F)}\lam(\set{x\in F: h(x)> H})$.
\end{corollary}
In practice, the way we prove escape of mass is by using the above corollary: We find a good fundamental domain for the unit
lattice $\psi(\cO_M^\times)$ on most of which we have control on the height.

Henceforth we restrict our discussion to dimension $n=3$. We now explain how to we choose good fundamental domains
for $\psi(\cO_M^\times)$ in which we control the height in a good enough manner. We need to introduce a few definitions first.
\begin{definition}
\label{simplex_set}A set $\Phi=\left\{ \alpha_{1},\alpha_{2},\alpha_{3}\right\} \subseteq\bR_{0}^{3}$
is called a \emph{simplex set} if $span_{\bR}\Phi=\bR_{0}^{3}$ and
$\sum_{1}^{3}\alpha_{i}=0$.

Denote by $\Delta_{\Phi}=span_{\bZ}\left\{ \Phi\right\} $ the lattice
generated by $\Phi$ and by $W_{\Phi}$ the set
\[
W_{\Phi}=\left\{ \sum_{1}^{3}\lambda_{i}\alpha_{i}\;\mid\;\left\{ \lambda_{1},\lambda_{2},\lambda_{3}\right\} =\left\{ 0,\frac{1}{3},\frac{2}{3}\right\} \right\} .
\]

\end{definition}
Since any simplex set $\Phi$ is a linear image of the simplex set giving rise to the 
regular triangle lattice, we conclude from Figure \ref{fig:A-fundamental-domain} that
$\on{conv}\left(W_{\Phi}\right)$ is a fundamental domain for the lattice $\Del_\Phi$.
\begin{figure}
\begin{centering}
\includegraphics[scale=0.65]{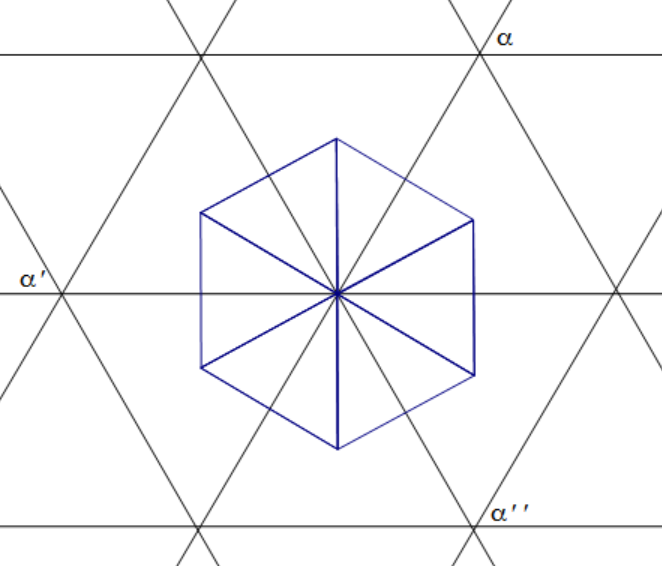}
\par\end{centering}

\caption{\label{fig:A-fundamental-domain}The blue hexagon, which equals $\on{conv}(W_\Phi)$, 
is a fundamental domain for the regular triangles lattice generated
by $\Phi=\{\alpha,\alpha',\alpha''\}$. The six points $W_\Phi$ are its vertices.}
\end{figure}

\begin{definition}
For a vector $v=\left(v^{(1)},v^{(2)},v^{(3)}\right)\in\bR_{0}^{3}$
we write $\left\lceil v\right\rceil =\max v^{(i)}$. For a set $\tilde{\Phi}$
we denote $\left\lceil \tilde{\Phi}\right\rceil ={\displaystyle \max_{v\in\tilde{\Phi}}}\left\lceil v\right\rceil $.\end{definition}
\begin{lemma}\label{lem:regdiv}
Let $\Phi_i$ be a sequence of simplex sets such that $\Phi_i\subset \psi(\cO_i^\times)$, where $\cO_i$ is a sequence
of distinct orders in totally real cubic fields. Then, $\lceil W_{\Phi_i}\rceil\to\infty.$
\end{lemma}
\begin{proof}
It is not hard to show that if the covolume of the lattice $\Del_\Phi$ goes to $\infty$ then so does $\lceil W_\Phi\rceil.$ The lemma now follows because the covolume of $\Del_{\Phi_i}$ is proportional to the regulator $R_{\cO_i}$ and it is well known that 
there are only finitely many orders with regulators under a given bound, the lemma follows.
\end{proof}
In Theorem~\ref{thm:partial escape} below we see that the term $\left\lceil W_{\Phi}\right\rceil $
controls the escape of mass. Before stating this theorem we need the following.
\begin{definition}\label{def:tight}
Let $M$ be a full module in a totally real cubic number field. We say that a simplex set 
$\Phi\subseteq\psi(\cO_M^\times)$ is $\left(R,r\right)$-\textit{tight}
for $R\geq1$ and $0\leq r\leq1$ if $\exp\left(r\left\lceil W_{\Phi}\right\rceil \right)\leq ht\left(L_M\right)R$.\end{definition}
\begin{theorem}\label{thm:partial escape}
\label{mass_escape} Let $R,r>0$ be fixed. 
Let $M_i$ be a sequence of full modules in totally real cubic number fields with distinct 
associated orders $\cO_i$. Let $\Phi_{i}\subseteq\psi(\cO_i^\times)$ be simplex sets which are
$\left(R,r\right)$-tight. Then the sequence of periodic $A$-orbits $AL_{M_i}$ exhibits $r^{2}$-escape of mass.
\end{theorem}
\begin{proof}
Consider the height function $h:\on{conv}(W_{\Phi_i})\to \bR$ given by $h(x)=ht(\exp(x)L_{M_i})$. We show that for any $H>0$, any
$r_0<r$, and any large enough $i$,  $r_0\cdot \on{conv}(W_{\Phi_i})\subset \set{x\in \on{conv}(W_{\Phi_i}): h(x)>H}$ and so, by Corollary~\ref{cor:escape}, since $r_0^2=\frac{\lam(r_0\cdot \on{conv}(W_{\Phi_i}))}{\lam(\on{conv}(W_{\Phi_i}))}$, the sequence exhibits 
$r^2$-escape of mass
as claimed.

To this end, fix $H>0, r_0<r$ and let $x\in r_{0}\cdot \on{conv}\left(W_{\Phi_{i}}\right)\subseteq\bR_{0}^{3}$
for some $0\leq r_{0}<r$ and write $x={\displaystyle \sum_{\beta\in W_{\Phi_{i}}}}\lambda_\be\beta$
as a convex combination. For $0\neq v\in L_{M_i}$ of norm $ht(L_{M_i})^{-1}$
we get that
\begin{eqnarray*}
\norm{\exp\left(x\right)v} & \leq & \norm v\max_{1\le \ell\le 3}\left(\exp\left(x_\ell\right)\right)\leq\norm v\max_{\ell}
(\exp(\sum_{\beta\in W_{\Phi_{i}}}\lambda_\be\beta^{(\ell)}))^{r_{0}}\\
 & \leq & 
 \norm v\left(\exp\left\lceil W_{\Phi_{i}}\right\rceil \right)^{r_{0}}\leq R\left(\exp\left\lceil W_{\Phi_{i}}\right\rceil \right)^{r_{0}-r}.
\end{eqnarray*}
It follows that for $x\in r_0\cdot \on{conv}(W_{\Phi_i})$, $h(x)>R^{-1} (\lceil W_{\Phi_{i}}\rceil )^{r-r_0}$ and the latter expression
is greater than $H$ for $i$ large enough since $\left\lceil W_{\Phi_{i}}\right\rceil \to\infty$ by
Lemma~\ref{lem:regdiv}.
\end{proof}

\section{\label{Construction-of-cubic}Construction of cubic orders}
\subsection{Generalizing the simplest cubic fields} To motivate the constructions presented below we begin by reviewing a
classical family of cubic fields. These are known as
the \emph{simplest cubic fields}.  They get their name from the ease in the
computation of their integer ring, integral units and other important
algebraic invariants. This example was used by Cusick in \cite{cusick_lower_1984}
to show that the limit $\limfi i{\infty}\frac{R_{i}}{\log^{2}\left(D_{i}\right)}=\frac{1}{16}$ in Theorem~\ref{thm:cusik}
can be attained. This family is defined by the polynomials
\[
f_t(x)=x^{3}-tx^{2}-(t+3)x-1=\left(x^{3}-3x-1\right)-t\cdot x\left(x+1\right)\quad t\in\bZ.
\]
It is well known that for infinitely many $t$ we have that $\cO_{\bK_{t}}=\bZ\left[\theta_{t}\right]$
where $\theta_{t}$ is a root of $f_{t}\left(x\right)$, and the unit
group is generated by $\theta_{t},\theta_{t}+1$. The fact that these
are indeed units is easy to see from the polynomials $f_{t}$. The
norm of $\theta_{t}$ is just the free coefficient of $f_{t}$, namely
$f_{t}(0)=-1$, so that $\theta_{t}$ is an integral unit. The norm
of $\theta_{t}+1$ is the free coefficient of $f_{t}(x-1)$, namely
it is $f_{t}(-1)=\left(-1\right)^{3}+3-1=1$ so it is again a unit.
Note that the norms are independent of $t$, since $0,-1$ are roots
of $x(x+1)$.

With this idea in mind, we construct below polynomials $f_{a,b,c,d,t}(x)$ giving rise to orders of the
form $\bZ\left[\theta\right]$ such that their unit group is generated
by $a\theta-b,\;c\theta-d$ and $-1$ (for $t$ large enough). These types of orders were
studied in \cite{minemura_totally_1998,thomas_fundamental_1979,grundman_systems_1995}
with some restrictions on $a,b,c,d$ and in greater generality in
\cite{minemura_totally_2004}, though with a rather complex set of conditions
on $a,b,c,d$. We will give a simple congruence condition on $a,b,c,d$ that will ensure that the group of
units is indeed generated by the above,
and furthermore, in \S\ref{many_examples} we will show how to construct
infinitely many tuples $(a,b,c,d)$ which satisfy our congruence conditions.\\

Given $a,b,c,d\in\bZ$ which satisfy some mild conditions, we classify the family of polynomials $f(x)\in\bZ[x]$ having a root $\theta$ such that
$a\theta-b, c\theta-d$ are units in the ring $\bZ[\theta]$.

\begin{lemma}\label{lem:unit_condition}
Let $f\left(x\right)\in\bZ\left[x\right]$ be a monic, cubic irreducible polynomial with root $\theta$.
Then for $a,b\in \bZ,\; a\neq 0$ we have that $N(a\theta-b)=-a^3f(\frac{b}{a})$. In particular,
$a\theta-b$ is a unit in $\bZ\left[\theta\right]$ if
and only if $a^3f(\frac{b}{a})=\pm1$. Additionally, if this is the case we
must have that $\gcd\left(a,b\right)=1$.\end{lemma}
\begin{proof}
We recall that $a\theta-b$ is a unit if and only if $N(a\theta-b)=\pm1$
and this norm is minus the free coefficient of the monic
minimal polynomial of $a\theta-b$ which is $a^{3}f\left(\frac{x+b}{a}\right)$.
It follows that $a\theta-b$ is a unit if and only if $\pm1=-N(a\theta-b)=a^{3}f\left(\frac{b}{a}\right)$.

A necessary condition is that $\left(a,b\right)=1$. Indeed, if $a=a'd$
and $b=b'd$ with $d>1$, then
\[
N(a\theta-b)=N(d\cdot\left(a'\theta-b'\right))=d^{3}N(a'\theta-b')\neq\pm1,
\]
and therefore $a\theta-b$ is not a unit.
\end{proof}
\begin{lemma}\label{lem:general form}
Let $a,b,c,d\in\bZ$, $\eps_1,\eps_2\in\set{\pm 1}$ be given and assume that $ad-bc\ne 0$ and that $\gcd(a,b)=\gcd(c,d)=1$.
Let
$$\cF=\cF_{a,b,c,d,\eps_1,\eps_2}=\set{h\in\bZ\left[x\right]:\begin{array}{lll}
\textrm{{\small
$h$ is a monic irreducible polynomial with root}}\\
\textrm{{\small
 $\theta$ such that both $a\theta-b$ and $c\theta-d$ are units}}\\
 \textrm{{\small
   in $\bZ[\theta]$ with norms $\eps_1,\eps_2$ correspondingly}}
 \end{array}}.$$
Then, if $h\in\cF$ we have $\cF \subseteq \set{h_{t}(x)=h(x)+t\left(ax-b\right)\left(cx-d\right): t\in\bZ}$.
\end{lemma}
\begin{proof}
If $f\in\cF$,
then $f(x)-h(x)$ is a degree 2 polynomial (since both are monic). Also,
from Lemma~\ref{lem:unit_condition} we conclude that
$\frac{b}{a}$, $\frac{d}{c}$ (which are distinct due to the hypothesis $ad-bc\ne 0$), are roots of $f-h$.
Then any such integer quadratic polynomials must be of the form $t\left(ax-b\right)\left(cx-d\right)$ for
some $t\in\bZ$, because of the primitivity assumption $\gcd(a,b)=\gcd(c,d)=1$. This establishes the inclusion
$\cF\subseteq \set{h_t(x):t\in\bZ}$.
\end{proof}
\begin{remark}
As the lemma above shows, the set $\{h_t(x)\; : \; t\in \bZ\}$ is exactly the set of cubic monic polynomials satisfying the conditions $a^3f(\frac{b}{a})=\pm1$ and $c^3f(\frac{d}{c})=\pm1$ appearing in Lemma~\ref{lem:unit_condition}. It is not true that any such polynomial is irreducible. For example $f(x)=x^2(x-2)+1$ satisfies $1^3f(\frac{0}{1})=f(0)=1$ and $1^3f(\frac{2}{1})=f(2)=1$, and it is not irreducible since it has a root $f(1)=0$. In Theorem~\ref{approx_roots} we shall see that these $h_t(x)$ are irreducible for all but finitely many $t$, and then Lemma~\ref{lem:unit_condition} will imply that the inclusion in the lemma above is cofinite.
\end{remark}
We note that at this point it is not obvious why $\cF\ne\varnothing$. Below we show that under suitable conditions on the parameters
$a,b,c,d$, this is indeed the case and moreover, our conditions will imply that the units $a\theta-b,c\theta-d$ generate (together with $-1$)
the group of units in $\bZ[\theta]$.

Recall that by Dirichlet's theorem, in a totally real cubic field, the unit group modulo its torsion part has rank $2$, and that two units are called a fundamental system if they generate the unit group modulo its torsion. If $a=0$, then $0\cdot\theta+b=b$ is a unit if and only if $b=\pm1$,
but of course it will not be a part of a fundamental system of units,
and therefore we must have that $a\neq0$, and similarly $c\neq 0$. On the other hand, if
$b=0$, then $a\theta$ can be a unit only when $a=\pm1$, namely
$\theta$ is a unit. If $d$ is also zero, then we get two units in
$\left\{ \pm\theta\right\} $ which cannot be a fundamental system.
We therefore assume in our discussion $a,c\neq0$ and at least one of
$b,d$ is nonzero.
\subsection{The case where $b=0$ or $d=0$}\label{zero_condition}
We analyze the case $d=0$, i.e.\ $a\theta-b, \theta$ form a fundamental set of units and the case $b=0$ is symmetric. We prove the following.
\begin{theorem}
\label{thm:one_unit}Let $\epsilon_{1},\epsilon_{2}\in\left\{ \pm1\right\} $
and $a,b\in\bZ\backslash\left\{ 0\right\} $ such that $\gcd\left(a,b\right)=1$.
There exists a monic polynomial $f(x)\in\bZ\left[x\right]$ such that
$a^{3}f\left(\frac{b}{a}\right)=\epsilon_{1}$ and $f\left(0\right)=\epsilon_{2}$
if and only if $a^{3}\equiv_{b}\epsilon_{1}\epsilon_{2}$ and $b^{3}\equiv_{a}\epsilon_{1}$.
In this case, there are infinitely many polynomials that satisfy this
condition and they have the form\footnote{Although this polynomial is supposed to be denoted by $f_{a,b,1,0,t}$
we omit the fixed parameters from the subscript to ease the notation.}
\begin{align}\label{eq:ht}
f_{a,b,t}(x)&=\\
\nonumber &\pa{x^{3}+\frac{\epsilon_{1}(a^{3}-\epsilon_{1}\epsilon_{2})^{2}-b^{3}}{ab^{2}}x^{2}-\epsilon_{1}a(\frac{a^{3}-\epsilon_{1}\epsilon_{2}}{b})x
+\epsilon_{2}}+t\cdot x(ax-b),
\end{align}
where $t\in\bZ$.
In particular, if $f$ is irreducible and $\theta:=\theta_{a,b,t}$ is a root of $f_{a,b,t}$,
then $\theta,a\theta-b$ are units in $\bZ\left[\theta\right]$.
\end{theorem}
\begin{proof}
The last sentence in the statement of the theorem follows directly from Lemma~\ref{lem:unit_condition}. 
Write $f(x)=x^{3}+Ax^{2}+Bx+C$ with $A,B,C\in\bZ$. The equation
$f(0)=\epsilon_{2}$ implies that $C=\epsilon_{2}$. The equation
$a^{3}f\left(\frac{b}{a}\right)=\epsilon_{1}$ translates to
\begin{eqnarray*}
b^{3}+Ab^{2}a+Ba^{2}b+\epsilon_{2}a^{3} & = & \epsilon_{1}\\
\left(A,B\right)\left(\begin{array}{c}
b\\
a
\end{array}\right) & = & \frac{\epsilon_{1}-b^{3}-\epsilon_{2}a^{3}}{ab}
\end{eqnarray*}
and hence $a,b\mid\epsilon_{1}-b^{3}-\epsilon_{2}a^{3}$ namely $a^{3}\equiv_{b}\epsilon_{1}\epsilon_{2}$
and $b^{3}\equiv_{a}\epsilon_{1}$. Assume now these congruence conditions
hold. The general solution to the equation above will be
\[
\left(A,B\right)=\left(\frac{\epsilon_{1}-b^{3}-\epsilon_{2}a^{3}}{ab^{2}},0\right)+s\left(\frac{a}{b},-1\right)=\left(\frac{\epsilon_{1}-b^{3}-\epsilon_{2}a^{3}+sa^{2}b}{ab^{2}},-s\right),
\]
and we need to show that we can choose an integer $s$ such that $A=\frac{\epsilon_{1}-b^{3}-\epsilon_{2}a^{3}+sa^{2}b}{ab^{2}}$
will be an integer. Since $\epsilon_{1}-b^{3}-\epsilon_{2}a^{3}+sa^{2}b\equiv_{a}\epsilon_{1}-b^{3}\equiv_{a}0$
and $\gcd\left(a,b^{2}\right)=1$, it is enough to show that this expression
is divisble by $b^{2}$. Because $\gcd\left(a^{2},b\right)=1$, we
can always find $s$ such that
\begin{align*}
sa^{2}\equiv_{b}\epsilon_{2}\left(\frac{a^{3}-\epsilon_{1}\epsilon_{2}}{b}\right)
&\iff\\
-sa^{2}b\equiv_{b^{2}}\epsilon_{2}\left(\epsilon_{1}\epsilon_{2}-a^{3}\right)
&\iff
-sa^{2}b\equiv_{b^{2}}\epsilon_{1}-b^{3}-\epsilon_{2}a^{3}
\end{align*}
Multiplying the first expresion by $a\epsilon_{1}\epsilon_{2}$ yields
$s\equiv_{b}\epsilon_{1}a\left(\frac{a^{3}-\epsilon_{1}\epsilon_{2}}{b}\right)$
so that $s=\epsilon_{1}a\left(\frac{a^{3}-\epsilon_{1}\epsilon_{2}}{b}\right)+bt$
for $t\in\bZ$. To complete the proof we note that
\begin{eqnarray*}
\left(A,B\right) & = & \left(\frac{\epsilon_{1}-b^{3}-\epsilon_{2}a^{3}+\epsilon_{1}\left(a^{3}-\epsilon_{1}\epsilon_{2}\right)a^{3}}{ab^{2}},-\epsilon_{1}a\left(\frac{a^{3}-\epsilon_{1}\epsilon_{2}}{b}\right)\right)+t\left(a,-b\right)\\
 & = & \left(\frac{-b^{3}+\epsilon_{1}\left(a^{3}-\epsilon_{1}\epsilon_{2}\right)^{2}}{ab^{2}},-\epsilon_{1}a\left(\frac{a^{3}-\epsilon_{1}\epsilon_{2}}{b}\right)\right)+t\left(a,-b\right).
\end{eqnarray*}
\end{proof}
\begin{remark}
Let $R=\left\{ \left(a,b\right)\in\bZ^{2}\;\mid\;\exists\epsilon_{1},\epsilon_{2}\in\left\{ \pm1\right\} ,\;a^{3}\equiv_{b}\epsilon_{1}\epsilon_{2},\;b^{3}\equiv_{a}\epsilon_{1}\right\} $.
Then clearly $R$ is symmetric (namely $\left(a,b\right)\in R\iff\left(b,a\right)\in R$)
and it is closed under multiplication by $(-1)$ for each coordinate
(namely $\left(a,b\right)\in R\iff\left(-a,b\right)\in R\iff\left(a,-b\right)\in R$).
In particular, up to these relations, we can look for examples such
that $\epsilon_{1},\epsilon_{2}=1$, i.e. $(a,b)$ is a mutually cubic root pair, and $\left|a\right|\leq\left|b\right|$. 
\end{remark}
\begin{example}
\label{theta_is_a_unit}
In this example we work out a simple recipe and show how to construct
an infinite family of mutually cubic root pairs $(a,b)$, and in particular,
to which Theorem~\ref{thm:one_unit} applies with $\epsilon_{1}=\epsilon_{2}=1$.
\begin{itemize}
\item The pairs $(a,1),(1,b)$ are always mutually cubic root pairs. In these cases the
polynomials are
\begin{eqnarray*}
f_{1,b,t}\left(x\right) & = & x^{3}-bx^{2}+\epsilon+tx\left(x-b\right)=x\left(x+t\right)\left(x-b\right)+1\\
f_{a,1,t}\left(x\right) & = & x^{3}+\left[\left(a\left(a^{3}-1\right)-a+t\right)x-1\right]\left[ax-1\right]\\
 & = & x^{3}+\left[sx-1\right]\left[ax-1\right]\quad;\quad s=a\left(a^{3}-1\right)-a+t
\end{eqnarray*}

\item Given any $b$, in order to solve $0\equiv_{a}b^{3}-1=\left(b-1\right)\left(b^{2}+b+1\right)$
we can choose $a=b^{2}+b+1$. The second
equation is satisfied automatically because $a^{3}=\left(b^{2}+b+1\right)^{3}\equiv_{b}1^{3}\equiv1$.
\item Similarly, given any $b$, we may take $a=1-b$ and get that the two congruences are satisfied.
\item Another option is to fix some integer $r$ and set $a=r^{2}$ and so
we have $a^{3}-1=\left(r^{3}\right)^{2}-1=\left(r^{3}-1\right)\left(r^{3}+1\right)$.
Thus, on choosing $b=r^{3}\pm1$ we get that $b^{3}\equiv_{a}\pm1$ and the other congruence condition follows as well.
\end{itemize}
\end{example}
We shall see in  \S\ref{many_examples} how to construct many
more examples.
\subsection{The case where both $b,d$ are non-zero}\label{nonzero_condition}

We claim that if we wish $a\theta-b,c\theta-d$ to be independent units in $\bZ[\theta]$ we need to assume $ad-bc\neq0$. 
Otherwise we would have that $c\theta-d=\frac{ad}{b}\theta-d=\frac{d}{b}\left(a\theta-b\right)$,
and because $N(a\theta-b),N\left(c\theta-d\right)=\pm1$ we would
get that $\frac{d}{b}=\pm1$. It follows that $a\theta-b=\pm\left(c\theta-d\right)$.

To make life easier, we will assume further that $ad-bc=\epsilon\in\left\{ \pm1\right\} $. We prove the following analogue of Theorem~\ref{thm:one_unit}.
\begin{theorem}
\label{theta_not_unit}
Let $\epsilon_{1},\epsilon_{2}\in\left\{ \pm1\right\} $
and $a,b,c,d\in\bZ\backslash\left\{ 0\right\} $ such that $ad-bc=\epsilon=\pm1$.
Then there exists a monic polynomial $f\left(x\right)\in\bZ\left[x\right]$
such that $a^{3}f\left(\frac{b}{a}\right)=\epsilon_{1}$ and $c^{3}f\left(\frac{d}{c}\right)=\epsilon_{2}$
if and only if $b^{3}\equiv_{a}\epsilon_{1}$ and $d^{3}\equiv_{c}\epsilon_{2}$.

In this case there are infinitely many polynomials that satisfy this
condition and they have the form $f_{a,b,c,d,t}\left(x\right)=x^{3}+Px^{2}+Qx+R$
with $t\in\bZ$ where
\begin{eqnarray*}
R & = & \epsilon\epsilon_{1}d^{3}-\epsilon\epsilon_{2}b^{3}+tbd\\
\left(P,Q\right) & = & \epsilon\left(\frac{\epsilon_{1}-b^{3}-Ra^{3}}{ab},\frac{\epsilon_{2}-d^{3}-Rc^{3}}{cd}\right)\left(\begin{array}{cc}
c & -d\\
-a & b
\end{array}\right).
\end{eqnarray*}
In particular, if $f_{a,b,c,d,t}(x)$ is irreducible and $\theta:=\theta_{a,b,c,d,t}$ is its root,
then $a\theta-b,c\theta-d$ are units in $\bZ\left[\theta\right]$.\end{theorem}
\begin{remark}
Note that although it is not apparent by the formula above, due to Lemma~\ref{lem:general form}, the cubic polynomials
arising in the above theorem are all of the form $h_0(x)+tg(x)$, where
$g=(ax-b)(cx-d)$.
\end{remark}
\begin{proof}
The last sentence in the statement of the theorem follows directly from Lemma~\ref{lem:unit_condition}.
We note first that $ad-bc\in\left\{ \pm1\right\} $ implies that $\gcd\left(a,b\right)=\gcd\left(a,c\right)=\gcd\left(d,b\right)=\gcd\left(d,c\right)=1$.

Writing $f(x)=x^{3}+Px^{2}+Qx+R$, we need to satisfy the equations
\begin{eqnarray*}
\epsilon_{1} & = & b^{3}+Pb^{2}a+Qba^{2}+Ra^{3}\quad\iff\quad ab\left(Pb+Qa\right)=\epsilon_{1}-b^{3}-Ra^{3}\\
\epsilon_{2} & = & d^{3}+Pd^{2}c+Qdc^{2}+Rc^{3}\quad\iff\quad cd\left(Pd+Qc\right)=\epsilon_{2}-d^{3}-Rc^{3}
\end{eqnarray*}
We conclude that $ab\mid\epsilon_{1}-b^{3}-Ra^{3}$, and since $a,b$
are coprime, this condition is equivalent to $b^{3}\equiv_{a}\epsilon_{1}$
and $Ra^{3}\equiv_{b}\epsilon_{1}$. Similarly we get that $d^{3}\equiv_{c}\epsilon_{2}$
and $Rc^{3}\equiv_{d}\epsilon_{2}$, thus proving the first direction
of the theorem.

Assume now that $b^{3}\equiv_{a}\epsilon_{1}$ and $d^{3}\equiv_{c}\epsilon_{2}$.
Since $\left(a,b\right)=\left(c,d\right)=1$, there are solutions
to $Aa^{3}\equiv_{b}\epsilon_{1}$ and $Cc^{3}\equiv_{d}\epsilon_{2}$.
Using $\left(b,d\right)=1$ and the Chinese remainder theorem, we
conclude that there is a solution $R\equiv_{b}A$ and $R\equiv_{d}C$
so that $Ra^{3}\equiv_{b}\epsilon_{1}$ and $Rc^{3}\equiv_{d}\epsilon_{2}$
and it is unique modulo $bd$. Once we have such an $R$ we get that
\begin{eqnarray*}
\left(P,Q\right)\left(\begin{array}{cc}
b & d\\
a & c
\end{array}\right) & = & \left(\frac{\epsilon_{1}-b^{3}-Ra^{3}}{ab},\frac{\epsilon_{2}-d^{3}-Rc^{3}}{cd}\right)\\
\left(P,Q\right) & = & \epsilon\left(\frac{\epsilon_{1}-b^{3}-Ra^{3}}{ab},\frac{\epsilon_{2}-d^{3}-Rc^{3}}{cd}\right)\left(\begin{array}{cc}
c & -d\\
-a & b
\end{array}\right)
\end{eqnarray*}
so that $P,Q$ are also integers, thus completing the first part of
the theorem.

Assume now that we have a solution to the above equations. By our
assumption, we have that
\begin{eqnarray*}
R & \equiv_{b} & R\left(bc+\epsilon\right)^{3}\epsilon=R\left(ad\right)^{3}\epsilon\equiv_{b}d^{3}\epsilon_{1}\epsilon\\
R & \equiv_{d} & R\left(ad-\epsilon\right)^{3}\left(-\epsilon\right)=R\left(bc\right)^{3}\left(-\epsilon\right)\equiv_{d}-b^{3}\epsilon_{2}\epsilon
\end{eqnarray*}
 Since $\left(b,d\right)=1$, the Chinese remainder theorem implies
that all the solutions have the form $R=\epsilon\epsilon_{1}d^{3}-\epsilon\epsilon_{2}b^{3}+tbd$
which completes the proof.\end{proof}
\begin{remark}
As in the previous case, given a suitable $a,b,c,d\in\bZ\backslash\left\{ 0\right\} $
in the theorem, we may multiply $a,b$ by $-1$ to get another pair
(which corresponds to the unit $-\left(a\theta-b\right)$). Thus,
we may assume that $\epsilon_{1}=1$ and similarly for $\epsilon_{2}=1$.
Furthermore, by switching the pairs $\left(a,b\right)$ with $\left(c,d\right)$,
we may assume that $ad-bc=1$. \end{remark}
\begin{example}
We choose $(a,b)$ as in the second bullet of Example~\ref{theta_is_a_unit}, i.e.\ $\left(a,b\right):=\left(b^{2}+b+1,b\right)$.
To find $(c,d)$ which solve the equation $ad-bc=1$
we choose for example $\left(c,d\right)=\left(b+1,1\right)$ and note that $d^{3}\equiv_{c}1^{3}=1$
so that the conditions of the theorem are satisfied.
\end{example}
The next lemma generalizes the example above and using the results
from \S\ref{many_examples} it produces infinitely many examples
of suitable tuples $a,b,c,d$ for the theorem above.
\begin{lemma}
Let $\left(a,c\right)$ be a pair such that $a^{3}\equiv_{c}1$ and
$c^{3}\equiv_{a}-1$. Then $\gcd\left(a,c\right)=1$ and the integers
$b,d$ such that $ad-bc=1$ satisfy $b^{3}\equiv_{a}1$ and $d^{3}\equiv_{c}1$.\end{lemma}
\begin{proof}
Since $ad-bc=1$ we get that
\begin{eqnarray*}
b^{3} & \equiv_{a} & -\left(bc\right)^{3}=-\left(ad-1\right)^{3}\equiv_{a}1\\
d^{3} & \equiv_{c} & \left(ad\right)^{3}=\left(1+bc\right)^{3}\equiv_{c}1
\end{eqnarray*}

\end{proof}

\subsection{\label{Full-escape-of}Full escape of mass}
Fix some integers $a,b,c,d$ which satisfy the conditions in Theorem~\ref{one_unit}
or Theorem~\ref{theta_not_unit}. Our goal in this section is to show that
if $\theta_{t}$ is a root of $h_t \defi f_{a,b,c,d,t}\left(x\right)$, then
the mass of the orbits $A\cdot L_{\bZ\left[\theta_{t}\right]}$ which
corresponds to the orders $\bZ\left[\theta_{t}\right]$ escape to
infinity as $\left|t\right|\to\infty$. Moreover, we shall show that
the shapes of the unit lattices in this family always converge to
the regular triangles lattice.

For that, we will need to find good approximations for the roots of
$h_{t}$. Heuristically, as $h_t=h_0+tg$ with $g=(ax-b)(cx-d)$, when $t$ is large,
$h_t$ will have roots close to the roots  $\frac{b}{a},\frac{d}{c}$ of $g$, and as it
is cubic, its third root will also be real. This simple idea is developed further below.
We shall use the following procedure. Since $f_{a,b,c,d,t}\left(\frac{b}{a}\right)=\pm\frac{1}{a^{3}}$
we will start with a guess that $\frac{b}{a}$ is close to the root.
We will then use Taylor expansion and the Newton Raphson method to
approximate the root.
\begin{theorem}
\label{thm:approx_roots}Let $h_{t}\left(x\right)$ be a familiy of
polynomials and $\alpha_t\in\bR$. Assume that
\begin{enumerate}
\item $h_{t}'\left(\alpha_{t}\right)\neq0$.
\item $\limfi t{\infty}\left|\frac{h_{t}(\alpha_{t})}{h_{t}'(\alpha_{t})}\right|=0$.
\item\label{ass229} $\limfi t{\infty}\left|\frac{h_{t}(\alpha_{t})}{h_{t}'(\alpha_{t})}\right|\left|\frac{h_{t}''\left(\alpha_{t}+\lambda\right)}{h_{t}'\left(\alpha_{t}\right)}\right|=0$
uniformly in $\left|\lambda\right|\leq1$. In particular this will
be true if $\left|\frac{h_{t}''\left(\alpha_{t}+\lambda\right)}{h_{t}'\left(\alpha_{t}\right)}\right|$
is uniformly bounded (in $t$ and $\left|\lambda\right|\leq1$).
\end{enumerate}
Then for $t$ large enough the $h_{t}$ have roots $\theta_{t}$ which
satisfy $$\theta_{t}=\alpha_{t}-\frac{h_{t}(\alpha_{t})}{h'_{t}(\alpha_{t})}+o\left(\left|\frac{h_{t}(\alpha_{t})}{h'_{t}(\alpha_{t})}\right|\right).$$

\end{theorem}
\begin{proof}
We start by getting a first approximation for the root. Letting $\varepsilon=-2\frac{h_{t}\left(\alpha_{t}\right)}{h_{t}'(\alpha_{t})}$
and using the Taylor expansion for $h_{t}$ we get that for some $\left|\lambda\right|\le1$
we have
\begin{align*}
h_{t}(\alpha_{t}+\varepsilon) & =  h_{t}(\alpha_{t})+h_{t}'(\alpha_{t})\varepsilon+\frac{h_{t}''(\alpha_{t}+\lambda\varepsilon)}{2}\varepsilon^{2}\\
&=h_{t}(\alpha_{t})+\varepsilon h_{t}'(\alpha_{t})\left(1+\frac{h_{t}''(\alpha_{t}+\lambda\varepsilon)}{2h_{t}'(\alpha_{t})}\varepsilon\right)\\
 & =  h_{t}(\alpha_{t})-2h_{t}(\alpha_{t})\left(1+\frac{h_{t}''(\alpha_{t}+\lambda\varepsilon)}{2h_{t}'(\alpha_{t})}\varepsilon\right)\\
 &= h_{t}\left(\alpha_{t}\right)\left[-1+2\frac{h_{t}''(\alpha_{t}+\lambda\varepsilon)}{h_{t}'(\alpha_{t})}\frac{h_{t}\left(\alpha_{t}\right)}{h_{t}'(\alpha_{t})}\right].
\end{align*}
For $t$ big enough $\left|\varepsilon\right|\leq1$ so that $\left|\varepsilon\lambda\right|\leq1$
hence we can use assumption~\eqref{ass229} to also assume that the term in the
brackets is negative. We conclude that $h_{t}\left(\alpha_{t}+\varepsilon\right),h_{t}\left(\alpha_{t}\right)$
have opposite signs and therefore $h_{t}$ has a root $\theta_{t}\in\left[\alpha_{t},\alpha_{t}+\varepsilon\right]$.

Applying the taylor expansion for $\theta_{t}$ and using $\left|\alpha_{t}-\theta_{t}\right|\leq\left|\varepsilon\right|=2\left|\frac{h_{t}\left(\alpha_{t}\right)}{h_{t}'(\alpha_{t})}\right|$
we get
\begin{eqnarray*}
0 & = & h_{t}\left(\alpha_{t}\right)+h_{t}'(\alpha_{t})\left(\theta_{t}-\alpha_{t}\right)+\frac{h_{t}''(\alpha_{t}+\lambda\varepsilon)}{2}\left(\theta_{t}-\alpha_{t}\right)^{2}
\end{eqnarray*}
\begin{eqnarray*}
\left|\left(\theta_{t}-\alpha_{t}\right)+\frac{h_{t}\left(\alpha_{t}\right)}{h_{t}'(\alpha_{t})}\right| & = & \left|\frac{h_{t}''(\alpha_{t}+\lambda\varepsilon)}{2h_{t}'(\alpha_{t})}\right|\left|\theta_{t}-\alpha_{t}\right|^{2}\leq4\left|\frac{h_{t}''(\alpha_{t}+\lambda\varepsilon)}{2h_{t}'(\alpha_{t})}\right|\left|\frac{h_{t}(\alpha_{t})}{h_{t}'(\alpha_{t})}\right|^{2}\\
 & = & 2\left|\frac{h_{t}''(\alpha_{t}+\lambda\varepsilon)h_{t}\left(\alpha\right)}{h_{t}'(\alpha_{t})h_t'\left(\alpha_{t}\right)}\right|\left|\frac{h_{t}(\alpha_{t})}{h_{t}'(\alpha_{t})}\right|=o\left(\left|\frac{h_{t}(\alpha_{t})}{h_{t}'(\alpha_{t})}\right|\right)
\end{eqnarray*}
and we are done.
\end{proof}
We now consider the case where $a,b,c,d$ are fixed and $t$ goes
to infinity.
\begin{theorem}\label{approx_roots}
Fix $a,b,c,d\in\bZ$ such that $\frac{b}{a}\neq\frac{d}{c}$ and there
exists a polynomial $h$ satisfying $a^{3}h\left(\frac{b}{a}\right)=\epsilon_{1},\;c^{3}h\left(\frac{d}{c}\right)=\epsilon_{2}$, where $\eps_i=\pm1$.
If $\frac{b}{a}+\frac{d}{c}\in\bZ$, we will further assume that $c^{3}\epsilon_{1}+a^{3}\epsilon_{2}\neq0$ (in particular this is true if $a\neq \pm c$).
We denote $h_{t}(x)=h(x)+tg(x),\quad g(x)=\left(ax-b\right)\left(cx-d\right)$
where $t\in\bZ$. Then the following holds
\begin{enumerate}
\item For $\left|t\right|$ big enough the polynomial $h_{t}\left(x\right)$
is totally real and irreducible.
\item The 3 roots of $h_{t}(x)$ satisfy
\begin{eqnarray*}
\theta_{1} & = & \frac{b}{a}+\Theta\left(\frac{1}{\left|a^{3}\left(bc-ad\right)t\right|}\right)\\
\theta_{2} & = & \frac{d}{c}+\Theta\left(\frac{1}{\left|c^{3}\left(bc-ad\right)t\right|}\right)\\
\theta_{3} & = & -act+O\left(1\right)
\end{eqnarray*}

\item The discriminant of $h_{t}$ is
\[
D_{h_{t}}=\left(\theta_{1}-\theta_{2}\right)^{2}\left(\theta_{2}-\theta_{3}\right)^{2}\left(\theta_{3}-\theta_{1}\right)^{2}=\left(\frac{b}{a}-\frac{d}{c}\right)^{2}\left(act\right)^{4}+O\left(t^{3}\right)
\]

\end{enumerate}
\end{theorem}
\begin{proof}
We first note that since $a,b,c,d$ are fixed we get that $h_{t}\left(x\right)=x^{3}+P_{t}x^{2}+Q_{t}x+R_{t}$
and
\begin{eqnarray*}
P_{t} & = & act+O\left(1\right)\\
Q_{t} & = & -\left(ad+bc\right)t+O\left(1\right)\\
R_{t} & = & bdt+O\left(1\right)
\end{eqnarray*}
Using these approximations and the hypothesis we get that
\begin{eqnarray*}
\left|h_{t}\left(\frac{b}{a}\right)\right| & = & \frac{1}{a^{3}},\quad h_{t}'\left(\frac{b}{a}\right)=2P_{t}\frac{b}{a}+Q_{t}+O\left(1\right)=\left(bc-ad\right)t+O\left(1\right)\\
h_{t}''\left(\frac{b}{c}+\lambda\right) & = & 2P_{t}+O\left(1\right)=act+O\left(1\right),\quad\forall\left|\lambda\right|\leq1
\end{eqnarray*}
It is now clear that
\begin{eqnarray*}
\left|\frac{h_{t}\left(b/a\right)}{h_{t}'\left(b/a\right)}\right| & = & \frac{1/a^{3}}{\left(bc-ad\right)t+O\left(1\right)}=\Theta\left(\frac{1}{\left|a^{3}\left(bc-ad\right)t\right|}\right)\to0\\
\left|\frac{h_{t}''\left(b/a+\lambda\right)}{h_{t}'\left(b/a\right)}\right| & = & \frac{act+O\left(1\right)}{\left(bc-ad\right)t+O\left(1\right)}\to\frac{ac}{bc-ad}
\end{eqnarray*}
Hence, we can use Theorem~\ref{thm:approx_roots} to approximate the root near
$\frac{b}{a}$ and similarly the roots near $\frac{d}{c}$ which are
\begin{eqnarray*}
\theta_{1} & = & \frac{b}{a}-\frac{h_{t}\left(b/a\right)}{h_{t}'\left(b/a\right)}+o\left(\frac{h_{t}\left(b/a\right)}{h_{t}'\left(b/a\right)}\right)\\
\theta_{2} & = & \frac{d}{c}-\frac{h_{t}\left(d/c\right)}{h_{t}'\left(d/c\right)}+o\left(\frac{h_{t}\left(d/c\right)}{h_{t}'\left(d/c\right)}\right).
\end{eqnarray*}
Note that since $\frac{b}{a}\neq\frac{d}{c}$, these two roots are
distinct for $\left|t\right|$ big enough, so that $h_{t}\left(x\right)$,
which is real of degree three, has at least two real roots, and therefore
has exactly three real roots.

We claim that for $|t|$ big enough, the roots of $h_t(x)$ are not integers. If $\frac{b}{a}\notin\bZ$, then for $\left|t\right|$ big enough
we see that $\theta_{1}\notin\bZ$. If $\frac{b}{a}\in\bZ$, then
for $\left|t\right|$ big enough $\theta_{1}$ can be an integer if
and only if it is $\frac{b}{a}$, but $h_{t}(\frac{b}{a})=\pm\frac{1}{a^{3}}\neq0$.
It follows that $\theta_{1},\theta_{2}\notin\bZ$ for $\left|t\right|$
big enough. Finally, since $\theta_{3}=-P_{t}-\theta_{1}-\theta_{2}$
and $P_{t}$ is an integer, we see that $\theta_{3}$ is an integer
if and only if $\theta_{1}+\theta_{2}$ is an integer. If $\frac{b}{a}+\frac{d}{c}\notin\bZ$,
then $\theta_{3}$ is not an integer for $\left|t\right|$ large enough.
If $\frac{b}{a}+\frac{d}{c}\in\bZ$, then we need to consider the
second approximation
\begin{align*}
&\frac{h(b/a)}{h'(b/a)+tg'(b/a)}+\frac{h(d/c)}{h'(d/c)+tg'(d/c)} \\
& = \frac{1}{act}\left[\frac{\epsilon_{1}}{a^{3}}\frac{1}{\frac{h'(b/a)}{act}+\left(\frac{b}{a}-\frac{d}{c}\right)}+\frac{\epsilon_{2}}{c^{3}}\frac{1}{\frac{h'(d/c)}{act}+\left(\frac{d}{c}-\frac{b}{a}\right)}\right]
\end{align*}
The limit of the expression inside the brackets is $\frac{1}{\left(\frac{b}{a}-\frac{d}{c}\right)}\left[\frac{\epsilon_{1}}{a^{3}}+\frac{\epsilon_{2}}{c^{3}}\right]$.
We assumed that in this case $\frac{\epsilon_{1}}{a^{3}}+\frac{\epsilon_{2}}{c^{3}}\neq0$
so that the root $\theta_{3}$ modulo $\bZ$ is $\Theta\left(\frac{1}{t}\right)+o\left(\frac{1}{t}\right)$
and in particular it is not an integer.

We showed that $h_{t}(x)$ doesn't have integer roots for $\left|t\right|$
big enough, and since it monic and has degree $3$, we conclude that
it is irreducible by using Gauss' lemma.

Finally, the approximation of the discriminant follows from the approximation
of the roots.
\end{proof}
Recall from Corollary~\ref{fundamental_units} that if $\frac{R'}{\log^{2}\left(D\right)}<\frac{1}{8}$
where $R'$ is the relative discriminant for some independent units,
then these units are actually a fundamental set. Using this, we are
able to show that $a\theta-b,\;c\theta-d$ are fundamental for $\left|t\right|$
big enough.
\begin{theorem}\label{thm:fixed_parameters}
Consider a family of polynomials $h_{t}\left(x\right)$ as in Theorem~\ref{approx_roots}
and choose a root $\theta^{(t)}$ for each $t$. Then
for $t$ large enough the unit group of $\bZ[\theta^{(t)}]$ is generated by $\left\{ a\theta^{(t)}-b,c\theta^{(t)}-d,-1\right\} $.
Furthermore, the shape of unit lattices of $\bZ\left[\theta^{\left(t\right)}\right]$
converge to the shape of the regular triangle lattice $\bZ\left[\omega\right]$ (where $\om=\exp(\frac{2\pi i}{3})$ and 
the correspondence is $\left(a\theta^{(t)}-b\right)\mapsto1$ and $\left(c\theta^{(t)}-d\right)\mapsto\left(1+\omega\right)$
) and the compact $A$-orbits of the lattices $L_t\in X$ corresponding to the orders $\bZ[\theta^{(t)}]$ exhibit full escape of mass.
\end{theorem}
\begin{proof}
The embedding of the units $a\theta^{(t)}-b,\;c\theta^{(t)}-d$ in
$\bR^{3}$ is
\begin{align*}
& \log\left(\left|a\theta^{(t)}-b\right|\right) \\
&=  \left(-\log\left|a^{2}\left(bc-ad\right)t\right|,\log\left|\frac{ad}{c}-b\right|,\log\left|a^{2}ct\right|\right)+O\left(1\right)\\
&=\log\left|t\right|\left(-1,0,1\right)+O\left(1\right)\\
&\log\left(\left|c\theta^{\left(t\right)}-d\right|\right) \\
& = \left(\log\left|\frac{cb}{a}-d\right|,-\log\left|c^{2}\left(bc-ad\right)t\right|,\log\left|ac^{2}t\right|\right)+O\left(1\right)\\
&=\log\left|t\right|\left(0,-1,1\right)+O\left(1\right)
\end{align*}

The relative regulator is $\log^{2}\left|t\right|+O\left(\log\left|t\right|\right)$
so that
\[
\frac{R_{t}'}{\log^{2}\left(D_{t}\right)}=\frac{\log^{2}\left(t\right)+O\left(\log\left|t\right|\right)}{\left(\log\left|\left(\frac{b}{a}-\frac{d}{c}\right)^{2}\left(act\right)^{4}\right|+O\left(1\right)\right)^{2}}=\frac{\log^{2}\left(t\right)+O\left(\log\left|t\right|\right)}{\left(4\log\left|t\right|+O\left(1\right)\right)^{2}}\to\frac{1}{16}.
\]
It follows that for $\left|t\right|$ big enough the units $a\theta^{(t)}-b,\;c\theta^{(t)}-d$
are a fundamental set.

We note that $\left(-1,0,1\right),\left(0,-1,1\right)$ generate the
regular triangles lattice. Indeed, the rotation around $\left(1,1,1\right)$
by $\frac{2\pi}{3}$ is just the cyclic permutation, and these two
vector are just the rotation of each other, up to a minus sign.

Using the simplex set 
$$\Phi=\left\{ \log\left|t\right|\left(-1,0,1\right),\log\left|t\right|\left(1,-1,0\right),\log\left|t\right|\left(0,1,-1\right)\right\} +O\left(1\right),$$
it is easily seen that $\left\lceil W_{\Phi}\right\rceil =\log\left|t\right|\frac{2}{3}+O\left(1\right)$.
On the other hand $D_{t}^{-1/6}\left(1,1,1\right)$ is in the normalized
unimodular lattice $L_{t}$ that correspond to $\bZ\left[\theta^{(t)}\right]$
so that $ht(L_{t})\geq\frac{1}{\sqrt{3}}D_{t}^{1/6}=\Theta\left(t^{2/3}\right)$.
We conclude that $\exp\left(\left\lceil W_{\Phi}\right\rceil \right)=O\left(t^{\frac{2}{3}}\right)\leq R\cdot ht\left(L_{t}\right)$
for some $R$ big enough and all $\left|t\right|$ big enough, namely
these orders are $\left(R,1\right)$-tight (see Definition~\ref{simplex_set}
for the definitions). It follows that there is a full escape of mass
by Theorem~\ref{mass_escape} which completes the proof.
\end{proof}
In the case of simplest cubic fields, the fundamental units are $\theta,\theta+1$.
It is known that for infinitely many $t$, the order $\bZ\left[\theta_{t}\right]$,
where $\theta_{t}$ is the root of $f_{t}(x)=x^{3}-3x+1-tx(x+1)$,
is the ring of integers of $\bQ\left(\theta_{t}\right)$. In particular
the $\bZ\left[\theta_{t}\right]$ belong to different field extensions.
We conclude that there are orbits coming from different fields such
that their mass escape to infinity.

While we do not have an example of orbits arising from the same field,
we can create long finite sequences of orbits such that most of their
mass is near the cusp.

Note that if the unit group is generated by $\left\langle na\theta-b,nc\theta-d\right\rangle $,
then $\bZ\left[\theta\right]$ and $\bZ\left[n\theta\right]$ have
the same unit group. Since $D_{n\theta}=n^{6}D_{\theta}$, the mass
of its corresponding orbit is farther away than the mass of $\bZ\left[\theta\right]$.
This leads to the following result.
\begin{theorem}
Let $1>\varepsilon>0$ and $K\subseteq SL_{3}(\bR)/SL_{3}(\bZ)$ be
a compact set. Then for each $N\in\bN$ we can find a sequence of
decreasing orders $\bZ\left[\theta_{1}\right]>\bZ\left[\theta_{2}\right]>\cdots>\bZ\left[\theta_{N}\right]$
with their corresponding orbits $A\cdot L_{1},...,A\cdot L_{N}$ such
that $\frac{\mu_{i}(A\cdot L_{i}\cap K)}{\mu_{i}(A\cdot L_{i})}<\varepsilon$
for each $i$ where $\mu_{i}$ is the induced $A$-invariant measure
on $A\cdot L_{i}$.\end{theorem}
\begin{proof}
Consider the polynomials $f_{t,n}(x)=x^{3}+t(2^{n}x-1)(2^{n-1}x-1)$
with corresponding roots $\theta_{t,n}$. From the previous theorem,
for a given compact set $K$ and $\varepsilon>0$ we can find $T$
big enough such that for all $t>T$ the $A$-orbits $A\cdot L_{t,n}$
corresponding to the orders $\bZ\left[\theta_{t,n}\right]$ satisfy
$\frac{\mu_{t,n}\left(A\cdot L_{t,n}\cap K\right)}{\mu_{t,n}\left(A\cdot L_{t,n}\right)}<\varepsilon$
for all $1\leq n\leq N$.

Notice that the minimal polynomial for $2\theta_{t,n}$ is
\begin{align*}
f_{t,n}\left(\frac{x}{2}\right)&=\frac{1}{8}x^{3}+t\left(2^{n-1}x-1\right)\left(2^{n-2}x-1\right)\\
&=\frac{1}{8}\left(x^{3}+8t\left(2^{n-1}x-1\right)\left(2^{n-2}x-1\right)\right)=\frac{1}{8}f_{8t,n-1}.
\end{align*}
It follows that $2\theta_{t,n}$ is a root of $f_{8t,n-1}$. Since
$\bZ\left[2\theta_{t,n}\right]=span\left\{ 1,2\theta_{t,n},4\theta_{t,n}^{2}\right\} $,
we see that $\left[\bZ\left[\theta_{t,n}\right]:\bZ\left[2\theta_{t,n}\right]\right]=8$,
so these are distinct orders. Using induction, we get the orders $\bZ\left[2^{N-1}\theta\right]<\bZ\left[2^{N-2}\theta\right]<\cdots<\bZ\left[2\theta\right]<\bZ\left[\theta\right]$
where $\theta:=\theta_{T,N}$, such that their corresponding orbits
$A\cdot L_{i}$ all satisfy $\frac{\mu_{i}\left(A\cdot L_{i}\cap K\right)}{\mu_{i}\left(A\cdot L_{i}\right)}<\varepsilon$.\end{proof}
\begin{problem}
Is there an infinite sequence of lattices coming from a fixed field,
or better yet, coming from a sequence of decreasing orders, which exhibits escape of mass?
\end{problem}

\subsection{\label{one_unit}The lattices $\protect\bZ\left[\theta\right]$,
where $\theta$ is a unit}

In the previous section, the shapes of the unit lattices converged
to the regular triangles lattice. In this section we show how to construct
more examples with different unit lattice shape.

We shall now confine our attention to analyze sequences of polynomials arising from Theorem~\ref{thm:one_unit} where the parameters $a,b$ are chosen 
as functions of $t$. More precisely, given a mutually cubic root sequence $(a_t,b_t)$ we define (as in~\eqref{eq:ht}) 
\begin{align}\label{eq:ht2}
h_{t}\left(x\right) & = f_{a_t,b_t,t}\left(x\right)=x^{3}+P_{t}x^{2}+Q_{t}x+1\\
\nonumber\left(P_{t},Q_{t}\right) & =  \left(\frac{\left(a^{3}-1\right)^{2}-b^{3}}{ab^{2}},-a\left(\frac{a^{3}-1}{b}\right)\right)+t\left(a,-b\right).
\end{align}
By Theorem~\eqref{thm:one_unit}, if $\theta_t$ is a root of $h_t$ then $\theta_t,a_t\theta_t-b_t$ are units of the order
$\bZ\left[\theta_t\right]$. 
In fact, we will make the following standing assumption that will help us in the analysis.
\begin{assumption}
\label{factAB} Let $h_t$ be the sequence of polynomials in~\eqref{eq:ht2} corresponding to  the mutually cubic root sequence $(a_t,b_t)$ 
and assume furthermore that
\begin{enumerate}
\item $\left|b_t\right|>\left|a_t\right|>0$ for each $t$, and
\item $\tilde{a}:=\limfi t{\infty}\frac{\log\left|a_{t}\right|}{\log\left(t\right)}$
and $\tilde{b}:=\limfi t{\infty}\frac{\log\left|b_{t}\right|}{\log(t)}$
exist and $\tilde{a}<\frac{1}{3},\;\tilde{b}<1$.
\end{enumerate}
\end{assumption}
\begin{remark}
The assumption above implies that $a_{t}^{r}=o(t)$ for $r\leq3$ and
that $b_{t}=o\left(t\right)$, so that $P_{t}=at+o\left(t\right)$
and $Q_{t}=-bt+o\left(t\right)$.
\end{remark}

We remark that some of the claims below are true in a more general
setting than the assumption above.
\begin{theorem}
Assume \ref{factAB}. Then the polynomial $h_{t}(x)$ is irreducible
over $\bQ$ for $\left|t\right|$ big enough.\end{theorem}
\begin{proof}
Note that since both the leading and free coefficient of $h_{t}$
are $\pm1$, we get that $h_{t}$ is reducible (over $\bQ$) if and
only if it has a root in $\pm1$.
\begin{eqnarray*}
h_{t}(\pm1) & = & O\left(1\right)+P_{t}\pm Q_{t}=\left(a\pm b\right)t+o\left(t\right)
\end{eqnarray*}
Since $b\neq\pm a$ are integers we conclude that $\left|h_{t}(\pm1)\right|\geq\frac{t}{2}$
for $\left|t\right|$ big enough, and hence $h_{t}\left(\pm1\right)\neq0$. \end{proof}
\begin{lemma}
Assume \ref{factAB}. Then
\begin{enumerate}
\item\label{1441} $\left|\frac{h_{t}''(\lambda)}{h_{t}'(0)}\right|$ is uniformly bounded
for $\left|\lambda\right|\leq1$ and $\left|\frac{h_{t}(0)}{h_{t}'(0)}\right|=\Theta\left(\frac{1}{\left|tb\right|}\right)\to0$.
\item\label{1442} $\left|\frac{h_{t}''(b/a+\lambda)}{h_{t}'(b/a)}\right|$ is uniformly
bounded for $\left|\lambda\right|\leq1$ and $\left|\frac{h_{t}(b/a)}{h_{t}'(b/a)}\right|=\Theta\left(\frac{1}{\left|a^{3}bt\right|}\right)\to0$.
\end{enumerate}
\end{lemma}
\begin{proof}
\eqref{1441}. We have the following:
\[
h_{t}(0)=\epsilon_{2}\quad,\quad\left|h_{t}'(0)\right|=\left|Q_{t}\right|=\Theta(\left|tb\right|)
\]
\[
\left|\lambda\right|\leq1\quad\Rightarrow\quad\left|h_{t}''(\lambda)\right|=\left|6\lambda+2P_{t}\right|=\Theta\left(\left|ta\right|\right)
\]
We conclude that $\left|h_{t}''(\lambda)\right|=\Theta\left(\left|ta\right|\right)=O\left(\left|h_{t}'(0)\right|\right)$
uniformly over $\left|\lambda\right|\leq1$ since $\left|a\right|<\left|b\right|$,
so that $\left|\frac{h_{t}''(\lambda)}{h_{t}'(0)}\right|$ is bounded.
Since $\left|h_{t}(0)\right|=1$ we have that $\left|\frac{h_{t}(0)}{h_{t}'(0)}\right|=\Theta\left(\frac{1}{\left|tb\right|}\right)\to0$

\eqref{1442}. The second claim is similar - we have
\[
h_{t}(\frac{b}{a})=\frac{\epsilon_{1}}{a^{3}},\quad\left|h'_{t}\left(\frac{b}{a}\right)\right|=\Theta\left(\left|bt\right|\right)\quad\Rightarrow\quad\left|\frac{h_{t}(b/a)}{h_{t}'(b/a)}\right|=\Theta\left(\frac{1}{\left|a^{3}bt\right|}\right)\to0
\]
Secondly, we have that
\begin{eqnarray*}
\left|\lambda\right|\leq1\;\Rightarrow\;\left|h_{t}''(b/a+\lambda)\right| & = & \left|6\left(\frac{b}{a}+\lambda\right)+2P_{t}\right|=\Theta(\left|ta\right|)\\
\left|\frac{h_{t}''\left(\frac{b}{a}+\lambda\right)}{h_{t}'\left(\frac{b}{a}\right)}\right| & = & \frac{\Theta\left(\left|ta\right|\right)}{\Theta\left(\left|tb\right|\right)}=O\left(\left|\frac{a}{b}\right|\right)=O\left(1\right)
\end{eqnarray*}
so that the expression above is bounded.
\end{proof}
\begin{corollary}
Assume \ref{factAB}. The roots and discriminant of $h_{t}(x)$ satisfy
\begin{eqnarray*}
\theta_{1} & = & 0-\frac{h_{t}(0)}{h_{t}'(0)}+o\left(\left|\frac{h_{t}(0)}{h_{t}'(0)}\right|\right)=\Theta\left(\frac{1}{\left|tb\right|}\right)\\
\theta_{2} & = & \frac{b}{a}-\frac{h_{t}(b/a)}{h_{t}'(b/a)}+o\left(\left|\frac{h_{t}(b/a)}{h_{t}'(b/a)}\right|\right)=\frac{b}{a}\left(1+\Theta\left(\frac{1}{a^{2}b^{2}t}\right)\right)\\
\theta_{3} & = & \Theta\left(\left|at\right|\right)\\
D_{f} & = & \Theta\left(t^{4}a^{2}b^{2}\right)
\end{eqnarray*}
\end{corollary}
\begin{proof}
The approximations for $\theta_{1},\theta_{2}$ follow from the previous
lemma and Theorem~\ref{thm:approx_roots}. The third root satisfies
\begin{eqnarray*}
\theta_{3} & = & -P_{t}-\theta_{1}-\theta_{2}.
\end{eqnarray*}
Note that $\theta_{1}\to0$ while $\theta_{2}=\Theta\left(\frac{b}{a}\right)=o\left(t\right)$
so that $\left|\theta_{3}\right|=\Theta\left(\left|P_{t}\right|\right)=\Theta\left(\left|at\right|\right)$.
Since $\theta_{1}=o(\theta_{2})$ and $\theta_{2}=o(\theta_{3})$
we get that the discriminant satisfies
\begin{eqnarray*}
D_{t} & = & \left(\theta_{1}-\theta_{2}\right)^{2}\left(\theta_{2}-\theta_{3}\right)^{2}\left(\theta_{3}-\theta_{1}\right)^{2}=\Theta\left(t^{4}a^{2}b^{2}\right).
\end{eqnarray*}

\end{proof}
We are now ready to show that $\theta,a\theta-b$ form a fundamental
set, compute the shape of the unit lattice and show that there is
a partial escape of mass. Recall Definition~\ref{def:tight} and Theorem~\ref{thm:partial escape}.
\begin{theorem}
\label{thm:tight_mass_escape}Assume \ref{factAB}. For $h_{t}$ as
above let $\theta_{t}$ be one of its roots and let $M_{t}=\bZ\left[\theta_{t}\right],\;L_{t}$
be the corresponding order and unimodular lattice. Then for $\left|t\right|$
big enough we have:
\begin{enumerate}
\item\label{tme1} The units $\left\{ \theta_{t},a_{t}\theta_{t}-b_{t}\right\} $ are
a set of fundamental units.
\item\label{tms2} There exist simplex sets for $M_t$ which are  $(R,r)$-tight for all $0\leq r<1$ which satisfies
$\frac{2}{3}\left(1-r\right)+\left(\frac{1}{3}-r\right)\left(\tilde{a}+\tilde{b}\right)>0.$
In particular, the family of compact $A$-orbits $AL_t$ exhibits partial escape of mass (by choosing
$r=\frac{1}{3}$) and there is a full escape if $\tilde{a},\tilde{b}=0$
(for example, for $a_{t},b_{t}$ bounded).
\end{enumerate}
\end{theorem}
\begin{proof}
Recall that we embedd the units into $\bR^{3}$ by sending a unit
$\alpha$ to $\left(\log\left|\sigma_{i}\left(\alpha\right)\right|\right)_{1}^{3}$
where $\sigma_{i}:\bZ\left[\theta_{t}\right]\to\bR$ are the three
real embeddings. Thus, using the previous corollary for approximating
the units $\theta_{t},a\theta_{t}-b$, we get that
\begin{eqnarray*}
\left(\log\left|\sigma_{i}\left(\theta_{t}\right)\right|\right)_{1}^{3} & = & \left(-\log\left|tb\right|,\;\log\left|\frac{b}{a}\right|,\;\log\left|at\right|\right)+O\left(1\right)\\
\left(\log\left|a\sigma_{i}\left(\theta_{t}\right)-b\right|\right)_{1}^{3} & = & \left(\log\left|b\right|,-\log\left|a^{2}bt\right|,\log\left|a^{2}t\right|\right)+O\left(1\right)
\end{eqnarray*}

We prove~\eqref{tme1}. The relative regulator $R_{i}'$ for these units is
\begin{align*}
\det&\left(
\smallmat{
-\log\left|tb\right| & \log\left|\frac{b}{a}\right|\\
\log\left|b\right| & -\log\left|a^{2}bt\right|
}
+O\left(1\right)
\right)\\
=&\log\left|tb\right|\log\left|ta^{2}b\right|-\log\left|b\right|\log\left|\frac{b}{a}\right|+O\left(\log\left(t\right)\right)
\end{align*}
so that
\begin{align*}
\frac{R_{i}'}{\log^{2}\left(D_{i}\right)} & =  \frac{\log\left|tb\right|\log\left|ta^{2}b\right|-\log\left|b\right|\log\left|\frac{b}{a}\right|+O\left(\log\left(t\right)\right)}{\log^{2}\left(t^{4}a^{2}b^{2}\right)+O\left(\log\left(t\right)\right)}\\
&\longrightarrow\frac{\left(1+\tilde{b}\right)\left(1+2\tilde{a}+\tilde{b}\right)-\tilde{b}\left(\tilde{b}-\tilde{a}\right)}{4\left(2+\tilde{a}+\tilde{b}\right)^{2}}.
\end{align*}
We claim that for $0\leq\tilde{a},\tilde{b}<1$, the expression above
is always smaller than $\frac{1}{8}$, and therefore the units $\theta_{t},a\theta_{t}-b$
form a fundamental set of units (see Corollary~\ref{fundamental_units}).
Indeed, the expression is strictly less than $\frac{1}{8}$ if and
only if
\begin{eqnarray*}
0 & \overset{?}{\leq} & 4\left(2+\tilde{a}+\tilde{b}\right)^{2}-8\left[\left(1+\tilde{b}\right)\left(1+2\tilde{a}+\tilde{b}\right)-\tilde{b}\left(\tilde{b}-\tilde{a}\right)\right]\\
 & = & 4\left(4+\tilde{a}^{2}+\tilde{b}^{2}+4\tilde{a}+4\tilde{b}+2\tilde{a}\tilde{b}\right)-8\left[1+2\tilde{b}+2\tilde{a}+3\tilde{a}\tilde{b}\right]\\
 & = & 8+4\tilde{a}^{2}+4\tilde{b}^{2}-16\tilde{a}\tilde{b}=4\left(\tilde{b}-\tilde{a}\right)^{2}+8\left(1-\tilde{a}\tilde{b}\right)
\end{eqnarray*}
This is clearly true if $0\leq\tilde{a},\tilde{b}\leq1$ and the equality
holds only if $\tilde{a}=\tilde{b}=1$.

We prove~\eqref{tms2}. Instead of working with $\theta_{t},a\theta_{t}-b$, we shall work
with the simplex set $\Phi=\left\{ \theta_{t},\theta_{t}^{-1}\left(a\theta_{t}-b\right),\left(a\theta_{t}-b\right)^{-1}\right\} $.
These units correspond to
\begin{align*}
&\left(\log\left|\sigma_{i}\left(\theta^{-1}\left(a\theta-b\right)\right)\right|\right)_{1}^{3} \\
&= -\left(-\log\left|tb\right|,\;\log\left|\frac{b}{a}\right|,\;\log\left|at\right|\right)+\left(\log\left|b\right|,\log\left|\frac{1}{a^{2}bt}\right|,\log\left|a^{2}t\right|\right)+O\left(1\right) \\
& =  \left(\log\left|tb^{2}\right|,-\log\left|ab^{2}t\right|,\log\left|a\right|\right)+O\left(1\right).\\
&\left(\log\left|\sigma_{i}\left(a\theta-b\right)^{-1}\right|\right)_{1}^{3} =  \left(-\log\left|b\right|,\log\left|a^{2}bt\right|,-\log\left|a^{2}t\right|\right)+O\left(1\right).
\end{align*}
The vertices of the fundamental domain $\on{conv}\left(W_{\Phi}\right)$
correspond to 
$$\theta^{\lambda_{1}}\left(a-b\theta^{-1}\right)^{\lambda_{2}}(\left(a\theta-b\right)^{-1})^{\lambda_{3}}$$
where $\left\{ \lambda_{1},\lambda_{2},\lambda_{3}\right\} =\left\{ 0,\frac{1}{3},\frac{2}{3}\right\} $
(see Definition~\ref{simplex_set}). From these vertices we need to find the
maximum of the coordinates. For example, on the first coordinate we
have $-\log\left|tb\right|,\;\log\left|tb^{2}\right|,\;-\log\left|b\right|$.
To get a maximum, we clearly need to assign the $\frac{2}{3}$ power
to $\log\left|tb^{2}\right|$ (which is positive) and $\frac{1}{3}$
to $-\log\left|b\right|$ (which is bigger than $-\log\left|tb\right|$),
hence obtaining $\log\left(\left|tb^{2}\right|^{2/3}\cdot\left|b\right|^{-1/3}\right)=\log\left(\left|t\right|^{\frac{2}{3}}\left|b\right|\right)$.
A similar computation for the second and third coordinate will produce
$\log\left(\left|a^{2}bt\right|^{\frac{2}{3}}\left|\frac{b}{a}\right|^{\frac{1}{3}}\right)=\log\left(\left|t\right|^{\frac{2}{3}}\left|ab\right|\right)$
and $\log\left(\left|at\right|^{\frac{2}{3}}\left|a\right|^{\frac{1}{3}}\right)=\log\left(\left|t\right|^{\frac{2}{3}}\left|a\right|\right)$.
It follows that the maximum is $\left\lceil W_{\Phi}\right\rceil =\log\left(\left|t\right|^{\frac{2}{3}}\left|ab\right|\right)+O\left(1\right)$.\\
The height of the unimodular lattice is controled by the size of $D^{-1/6}\left(1,1,1\right)$,
so that $ht\left(L_{i}\right)=\Theta\left(t^{\frac{2}{3}}\left|ab\right|^{\frac{1}{3}}\right)$.
The $(R,r)$-tightness condition is
\[
\exp\left(r\left\lceil \tilde{\Phi}\right\rceil \right)\leq R\cdot ht\left(L_{i}\right)\quad\iff\quad-\log\left(R\right)\leq\log\left|ht\left(L_{i}\right)\right|-r\left\lceil \tilde{\Phi}\right\rceil
\]
so it is enough to show that $\log\left|ht\left(L_{i}\right)\right|-r\left\lceil \tilde{\Phi}\right\rceil \to\infty$.
\begin{align*}
\log\left|ht\left(L_{i}\right)\right|-r\left\lceil \tilde{\Phi}\right\rceil &=\log\left(t^{\frac{2}{3}}\left|ab\right|^{\frac{1}{3}}\right)-r\log\left(\left|t\right|^{\frac{2}{3}}\left|ab\right|\right)+O\left(1\right)\\
&=\log\left|t\right|\left(\frac{2}{3}\left(1-r\right)+\left(\frac{1}{3}-r\right)\frac{\log\left|ab\right|}{\log\left|t\right|}\right)+O\left(1\right)
\end{align*}
Thus, by taking $\left|t\right|\to\infty$, we see that the condition is equivalent
to
\[
\frac{2}{3}\left(1-r\right)+\left(\frac{1}{3}-r\right)\left(\tilde{a}+\tilde{b}\right)>0.
\]
We immediately see that if $r\leq\frac{1}{3}$, then this condition
is always satisfied and therefore we always have partial escape of
mass. On the other extreme, if $\tilde{a},\tilde{b}=0$ (for example
if $a_{t},b_{t}$ are bounded), then the inequality is true for all
$r<1$, so that we have a full escape of mass.
\end{proof}
\begin{theorem}
\label{thm:compute_shapes}Assume \ref{factAB}. For $h_{t}$ as above
let $\theta_{t}$ be one of its roots and let $M_{t}=\bZ\left[\theta_{t}\right],\;L_{t}$ be
the corresponding order and unimodular lattice. Let $\left[z_{t}\right]\in SL_{2}(\bZ)\backslash\mathbb{H}$
be the shape of the unit lattice $\psi\left(\bZ\left[\theta_{t}\right]^{\times}\right)\subset
\bR_{0}^{3}$. Then the sequence $\left[z_{t}\right]$ converges
to some point $\left[z\right]\in \SL_{2}\left(\bZ\right)\backslash\mathbb{H}$
where $z\in\mathbb{H}$ satisfies the following:

\begin{itemize}
\item $z(\tilde{a},\tilde{b})=\frac{1+2\tilde{a}+\left(1+\tilde{b}+2\tilde{a}\right)\omega}{1+\tilde{a}+\left(\tilde{a}-\tilde{b}\right)\omega}$
where $\omega=\frac{-1+\sqrt{3}i}{2}$ is a primitive root of unity
of order 3.
\item If $\tilde{a}=\tilde{b}=0$, then $z=1+\omega$, or equivalently the
lattice shape is the regular triangles lattice.
\item If $\tilde{a}=0$ then $\left|z\right|=1$.
\item If $\tilde{a}=\tilde{b}$, then $Re(z)=\frac{1}{2}$.
\end{itemize}

Moreover, if $0<\tilde{a}<\tilde{b}$ are small enough (say, $<\frac{1}{10}$),
then $z$ is in the interior of the standard fundamental domain of
$\SL_{2}(\bZ)$ in $\mathbb{H}$.

\end{theorem}
\begin{proof}
From the previous theorem, the unit lattice is generated by the elements
\begin{align*}
&\smallmat{
\log\left|\sigma_{i}\left(\theta_{t}\right)\right|\\
\log\left|a\sigma_{i}\left(\theta_{t}\right)-b\right|}_{1}^{3} \\
&=  \log\left(t\right)
\smallmat{
-1 & 0 & 1\\
0 & -1 & 1
}
+\log\left|b\right|
\smallmat{
-1 & 1 & 0\\
1 & -1 & 0
}
+\log\left|a\right|
\smallmat{
0 & -1 & 1\\
0 & -2 & 2
}
+O\left(1\right)\\
 & =  \log\left(t\right)\left[
\smallmat{
-1 & 0 & 1\\
0 & -1 & 1
}
+\frac{\log\left|b\right|}{\log\left|t\right|}
\smallmat{
-1 & 1 & 0\\
1 & -1 & 0
}
+\frac{\log\left|a\right|}{\log\left|t\right|}
\smallmat{
0 & -1 & 1\\
0 & -2 & 2
}
+O\left(\frac{1}{\log\left|t\right|}\right)\right]
\end{align*}
Note that the vectors $\left(\begin{array}{ccc}
-1 & 0 & 1\end{array}\right)$ and $\left(\begin{array}{ccc}
0 & -1 & 1\end{array}\right)$ have the same norm and have angle $\frac{\pi}{3}$ between them,
so we can define a similarity from $\bR^3_0$ to $\bC$ by sending them to $1,1+\omega$
respectively. We thus get lattice in $\bC$ having the same shape which is generated (in the limit as $\av{t}\to\infty$) by
\begin{eqnarray*}
v & = & 1-\tilde{b}\omega+\tilde{a}(1+\omega)=1+\tilde{a}-\left(\tilde{b}-\tilde{a}\right)\omega\\
u & = & 1+\omega+\tilde{b}\omega+2\tilde{a}(1+\omega)=1+2\tilde{a}+\left(1+\tilde{b}+2\tilde{a}\right)\omega.
\end{eqnarray*}
This lattice has the same shape as the one generated by $1,\frac{u}{v}$ which is the claim in the first bullet of the theorem.

If $\tilde{a}=0$, then $\frac{u}{v}=\frac{1+\left(1+\tilde{b}\right)\omega}{1-\tilde{b}\omega}$
and then $\av{\frac{u}{v}}^{2}=\frac{1-\left(1+\tilde{b}\right)+\left(1+\tilde{b}\right)^{2}}{1+\tilde{b}+\tilde{b}^{2}}=1$.
Similarly, if $\tilde{a}=\tilde{b}$, then $\frac{u}{v}=\frac{1+2\tilde{a}+\left(1+3\tilde{a}\right)\omega}{1+\tilde{a}}$
so that
\[
Re\left(\frac{u}{v}\right)=\left(\frac{1+2\tilde{a}}{1+\tilde{a}}\right)+\left(\frac{1+3\tilde{a}}{1+\tilde{a}}\right)\left(\frac{-1}{2}\right)=\frac{1}{2}.
\]
We leave it as an exercise to show that for small $0\leq\tilde{a},\tilde{b}$
the number $\frac{u}{v}$ is inside the standard fundamental domain
and is strictly inside if $0<\tilde{a}<\tilde{b}$.

\end{proof}
We now have all we need in order to prove Theorem \ref{thm:curves}.
\begin{proof}[Proof of Theorem \ref{thm:curves}]
Let $(a_t,b_t)$ be a mutually cubic root sequence and suppose that the limits $\tilde{a}=\limfi t{\infty}\frac{\log\left|a_t\right|}{\log\left|t\right|}$ and $\tilde{b}=\limfi t{\infty}\frac{\log\left|b_t\right|}{\log\left|t\right|}$ exist and satisfy $0\leq \tilde{a}\leq \tilde{b}$.
Given $p,q\in \bN$, we reindex the sequence $(a_t,b_t,t)$ and consider the sequence $(a_{t^p},b_{t^p},t^q)$. Obviously, this is also a mutually cubic root sequence, and the corresponding limits are $\frac{p}{q}\tilde{a}$ and $\frac{p}{q}\tilde{b}$.  In particular, taking $r=\frac{p}{q} \in \left[0, \min(\frac{1}{3\tilde{a}},\frac{1}{\tilde{b}}) \right) \cap\mathbb{Q}$ we get a sequence that satisfies Assumption \ref{factAB}, hence we get the limit point $z(r\tilde{a},r\tilde{b})=\frac{1+2r\tilde{a}+\left(1+r\tilde{b}+2r\tilde{a}\right)\omega}{1+r\tilde{a}+\left(r\tilde{a}-r\tilde{b}\right)\omega}$ in $\overline {\Omega}$, where $\Omega$ was defined to be the set of shapes of unit lattices. Finally, using the continuity of $z(x,y)$ it follows that $z(r\tilde{a},r\tilde{b})\in \overline{\Omega}$ for all $r \in \left[0, \min(\frac{1}{3\tilde{a}},\frac{1}{\tilde{b}}) \right]$.
\end{proof}

 \subsection{\label{many_examples}Finding the limits of $ \limfi t{\infty}\frac{\log\left|a_{t}\right|}{\log\left|b_{t}\right|}$}

As Theorem~\ref{thm:curves} shows, once we are given a mutually cubic root sequence $(a_t,b_t)$, the parameter that controls the curve is the ratio $\frac{\tilde{a}}{\tilde{b}}= \limfi t{\infty}\frac{\log\left|a_{t}\right|}{\log\left|b_{t}\right|}$ and we are left with the task of finding such sequences inducing different ratios. Let us denote by $\Lambda$ the set of such limits inside $P^1(\bR)$ (ignoring the case where $\tilde{a}=\tilde{b}=0)$, so that any $0\leq \lambda \leq 1$ in $\Lambda$ corresponds to a curve in $\overline{\Omega}$.

Before we turn to study the set $\Lambda$, let us concentrate on the case where $\tilde{a}=0$. Since $b_t \mid a_t^3 -1$, unless $a_t=1$ for all $t$ big enough, we will also get that $\tilde{b}\leq 3\tilde{a}=0$, which by Theorem \ref{thm:compute_shapes} implies that the shapes of the unit lattices converge to the regular triangle lattices (which is like the case discussed in Section \ref{Full-escape-of}).

Assuming now that $a_t=1$ for all $t$, Theorem~\ref{thm:one_unit} tell us that
\[
f_{t}\left(x\right)=\left(x^{3}-b_{t}x^{2}+1\right)+t\cdot x\left(x-b_{t}\right)=x\left(x-b_{t}\right)\left(x+t\right)+1.
\]
This type of polynomials was already studied by Cusick in \cite{cusick_regulator_1991}
where he showed that the limit points of the shapes of unit lattices is
on $\left|z\right|=1$ in the hyperbolic plane. This follows readily
from our computations in the previous section if $b_{t}=o\left(t\right)$ and in addition
we know that there is always a partial escape of mass. Moreover,
when $b_{t}\sim t^{\alpha},\;\alpha<1$, the cosine of the angle of
the corresponding point on the hyperbolic plane is $\frac{1-2\alpha-2\alpha^{2}}{2+2\alpha+2\alpha^{2}}$.
In particular the angle is $\frac{\pi}{3}$ when $\alpha=0$ and it
increases up until $\frac{2\pi}{3}$ when $\alpha\to1^{-}$. In case
that $b_{t}=Bt$ for some constant $B$, our analysis doesn't
hold. This case falls into the settings studied in ~\cite{Shapira} where it
was shown that there is a full escape of mass and that the unit lattices
shapes converge to the regular triangles lattice.

Let us continue to the general case of an element in $\Lambda$.  We have already seen several examples in Example~\ref{theta_is_a_unit}
of mutually cubic root sequences producing the limits $0,\infty,\frac{1}{2},\frac{1}{3},\frac{2}{3}\in\Lambda$,
so that $\Lambda$ is not empty. On the other hand, if $a_{t},b_{t}\neq\pm1$
and $b_{t}^{3}\equiv_{a_{t}}1$, then $\left|a_{t}\right|\mid\left|b_{t}^{3}-1\right|$
so that $\limfi t{\infty}\frac{\log\left|a_{t}\right|}{\log\left|b_{t}\right|}\leq\limfi t{\infty}\frac{\log\left|b_{t}^{3}-1\right|}{\log\left|b_{t}\right|}\to3$
and reversal of the roles of $a_{t},b_{t}$ produces a lower bound
$\frac{1}{3}$, so that $\Lambda\subseteq\left[\frac{1}{3},3\right]\cup\{0,\infty \}$.
\begin{lemma}
\label{lem:operations_on_lambda}We have the following:
\begin{enumerate}
\item\label{2001} The set $\Lambda$ is closed under taking inverses.
\item\label{2002} If $s\in\Lambda$, then $3-s\in\Lambda$.
\end{enumerate}
\end{lemma}
\begin{proof}
\eqref{2001}. Clearly, any mutually cubic root sequence $\left(a_{t},b_{t}\right)$ produces another such
sequence $\left(b_{t},a_{t}\right)$ so that $s\in\Lambda$ if and
only if $s^{-1}\in\Lambda$. Note that on the level of units, this
is nothing more than considering the fundamental units $\left\{ \theta^{-1},-\theta^{-1}\left(a\theta-b\right)\right\} $
instead of $\left\{ \theta,a\theta-b\right\} $.

\eqref{2002}. Let $\left(a,b\right)$ be a mutually cubic roots pair and suppose first that $b\neq 1$. Setting $c=\frac{1-b^{3}}{a}\neq 0$ we get that $c\mid1-b^{3}$
so that $b^{3}\equiv_{c}1$. On the other hand we have that
\[
c^{3}\equiv_{b}\left(ca\right)^{3}=\left(1-b^{3}\right)^{3}\equiv_{b}1
\]
 so that $\left(c,b\right)$ is another mutually cubic root pair. Taking the limit
we get that $\frac{\tilde{c}}{\tilde{b}}=\frac{3\tilde{b}-\tilde{a}}{\tilde{b}}=3-\frac{\tilde{a}}{\tilde{b}}$.

If on the other hand $b_t=1$ for almost all $t$, then $\frac{\tilde{a}}{\tilde{b}}=\infty$ (since we assumed that $(\tilde{a},\tilde{b})\neq(0,0)$), and then $3-\infty = \infty \in \Lambda$, hence the claim is still true.
\end{proof}
Both of the maps $s\to s^{-1}$ and $s\to3-s$ have order 2, but their
composition $T(s)=3-\frac{1}{s}$ has infinite order and acts on $P^1(\bR)$. We start with some basic properties of this M\"{o}bius action.
\begin{lemma}
\label{lem:mobius}
Let $T(s)=3-\frac{1}{s}$. Then:
\begin{enumerate}
\item The fixed points of $T$ are $\alpha_\pm=\frac{3\pm\sqrt{5}}{2}$ where $\frac{1}{3}<\alpha_-<\alpha_+<3$. In addition, any other $T$ orbit is inifinite.
\item $T[(\alpha_-,\alpha_+)]=(\alpha_-,\alpha_+)$  and $T(s)>s$ in this segment. 
\item $T[(\alpha_+,\infty]]= (\alpha_+,3]$ and $T(s)<s$ for $s\in (\alpha_+,\infty]$.
\item $T^{-1}[ [0,\alpha_-) ] =[\frac{1}{3},\alpha_-)$ and $T(s)<s$ for $s\in [\frac{1}{3},\alpha_-)$.
\item $T$  satisfies $T(\frac{1}{3})=0,\;T(0)=\infty$ and $T(\infty)=3$.
\item The only accumulation points of a single $T$-orbit in $\Lambda$ are $\alpha_\pm$.
\end{enumerate}
\end{lemma}
\begin{proof}
Left as an exercise.
\end{proof}
\begin{corollary} $\Lambda$ is infinite.
\end{corollary}
\begin{proof}
Since $3,\frac{1}{3},\frac{1}{2}\in\Lambda$, the lemma above provide
2 infinite $T$-orbits in $\Lambda$, and in particular $\Lambda$ is infinite in itself.
\end{proof}

It is now easily seen that $\Lambda$ contains at least two accumulation points at $\frac{3\pm \sqrt{5}}{2}$ which are exactly the fixed points of $T$. Define $\tilde{T}$ to be the corresponding action on sequences of mutually cubic roots, i.e. $\tilde{T}(a_t,b_t)=(\frac{1-a_t^3}{b_t},a_t)$. This is a composition of switching the sequences $(a_t,b_t)\mapsto (b_t,a_t)$ and then using the fact that $b_t^3-1\equiv _{a_t} 0$ we map $(b_t,a_t)\mapsto (\frac{1-a_t^3}{b_t},a_t)$. 

Suppose now that $b_t \mid a_t^2+a_t+1$, e.g. the sequence $(t, 1)$. In this case we will get that $(1-a_t)b_t \mid a_t^3-1$ so we may define the operation $\tilde{D}(a_t,b_t)=(a_t,(1-a_t)b_t)$ and hope to get a new mutually cubic root pair. As the next lemma shows, while we cannot use this operation on any sequence of mutually cubic root pairs, there are enough such sequences.
\begin{lemma}
Define $\tilde{D}:(\bZ -\{1\})\times \mathbb{Z} \to (\bZ -\{1\})\times \mathbb{Z} $ by $D(a_t,b_t)=(a_t,(1-a_t)b_t)$. Then the following holds:
\begin{enumerate}
\item\label{2131} The map $\tilde{D}$ induces a bijection between the set of mutually cubic root pairs $(a,b)$ with $b \mid a^2+a+1$ 
and the set of mutually cubic root pairs $(c,d)$ with $(1-c)\mid d$.
\item\label{2132} Suppose that $(a_t,b_t)$ is a sequence mutually cubic root pairs satisfying $|a_t| \to \infty$ and $b_t \mid a_t^2+a_t+1$ for each $t$ (namely, we can apply $\tilde{D}$ to it). If $s:=\lim_{t\to \infty}\frac{\log|a_t|}{\log|b_t|}>\frac{3+\sqrt{3}}{2}$ (including $\infty$), then $(c_t,d_t)=\tilde{T} \circ \tilde{T} \circ \tilde{D}(a_t,b_t)$ is also a sequence mutually cubic root pairs satisfying $d_t \mid c_t^2+c_t+1$ and $|c_t|\to \infty$. More over, we have
$$R(s):=\lim_{t\to \infty}\frac{\log|c_t|}{\log|d_t|}=\frac{5s-3}{2s-1}$$
and $\frac{3+\sqrt{3}}{2}<R(s)<s$.
\end{enumerate}
\end{lemma}
\begin{proof}
\eqref{2131}. Suppose first that $(a,b)$ is a mutually cubic roots pair with $b \mid a^2+a+1$. We clearly have that $(c,d):=(a,(1-a)b)$ satisfy $(1-c)\mid d$ and $c^3-1\equiv_{d} 0$. Furthermore, we have that $d^3=(1-a)^3b^3 \equiv_{a} b^3\equiv _{a}1$ so that $(c,d)$ is a mutually cubic root pair. On the other hand, if $(c,d)$ is a mutually cubic root pair with $(1-c)\mid d$, then we similarly get that $(c,\frac{d}{1-c})$ is again a mutually cubic root pair satisfying $b \mid a^2+a+1$ and $\tilde{D}(a,b)=(c,d)$.

\eqref{2132}. Assuming that $(a_t,b_t)$ is a mutually cubic root pair satisfying  $b_t \mid a_t^2+a_t+1$ we get that
\begin{eqnarray*}
(c_t,d_t) & = & \tilde{T} \circ \tilde{T} \circ \tilde{D}(a_t,b_t) = \tilde{T} \circ \tilde{T} (a_t,(1-a_t)b_t)=
\tilde{T}  \left(\frac{1-a_t^3}{(1-a_t)b_t},a_t\right) \\
 & = &\tilde{T}  \left(\frac{a_t^2+a_t+1}{b_t},a_t\right) 
=\left(\left(1-\left[\frac{a_t^2+a_t+1}{b_t}\right]^3\right)/a_t,\frac{a_t^2+a_t+1}{b_t}\right)
\end{eqnarray*}

We assumed that $b_t \mid a_t^2+a_t+1$ and $|a_t|\to \infty$ so that in particular $b_t,a_t,1-a_t\neq 0$ for almost all $t$ and we can divide by them. It follows that $(c_t,d_t)$ are well defined, and from part~\eqref{2131} it is also a  mutually cubic root pair.

Since $s=\lim_{t\to \infty}\frac{\log|a_t|}{\log|b_t|}>\frac{3+\sqrt{3}}{2}>\frac{1}{2}$ and $|a_t|\to \infty$, we conclude that $\left| \frac{a_t ^2+a_t+1}{b_t}\right| \to \infty$ and therefore
$$R(s):=\lim_{t\to \infty} \frac{\log|c_t|}{\log|d_t|}=\frac{5s-3}{2s-1}.$$

It is straight forward to show that if $s>\frac{3+\sqrt{3}}{2}$, then $s>R(s)>\frac{3+\sqrt{3}}{2}$. It follows that since
$|d_t|=\left| \frac{a_t ^2+a_t+1}{b_t} \right| \to \infty$ we must also have that $|c_t| \to \infty$. Finally, we need to show that $c_t^2+c_t+1\equiv_{b_t} 0$. Indeed, since $\gcd(a_t,d_t)=1$ we have $c_t \equiv _{d_t} (1-d_t^3) \cdot \frac{1}{a_t} \equiv _{d_t} \frac{1}{a_t}$ so that
$$c_t^2+c_t+1\equiv_{d_t} \frac{1+a_t+a_t^2}{a_t^2} = \frac{1+a_t+a_t^2}{b_t}\frac{b_t}{a_t^2}\equiv_{d_t}0.$$
\end{proof}

\begin{corollary}\label{cor:accpts}
There are infinitely many $T$-orbits in $\Lambda$ in the open interval $\left(\frac{3-\sqrt{5}}{2},\frac{3+\sqrt{5}}{2}\right)$
and $\Lambda$ has infinitely many accumulation points.\end{corollary}
\begin{proof}
Consider the sequence $(a_t,b_t)=(t,1)$ of mutually cubic root pairs. This sequence satisfies the conditions from the previous lemma, i.e. that $1=b_t \mid a_t^2 +a_t+1$ and that $|a_t|=|t|\to \infty$ as $t\to \infty$. Furthermore, we have that $s=\lim \frac{\log |a_t|}{\log|b_t|} = \infty > \frac{3+\sqrt{3}}{2}$ so from the previous lemma, for every $n$ the sequence $[(\tilde{T}^2\circ \tilde{D})^n (t,1)]_{t=1} ^\infty$ is a sequence of mutually cubic root pairs, and it corresponds to the limit $R^n (\infty) \in \Lambda$, where $R(s)$ is defined in the previous lemma. This sequence is decreasing and converges to $\frac{3+\sqrt{3}}{2}$ which is in the segment $\left( \frac{3-\sqrt{5}}{2}, \frac{3+\sqrt{5}}{2} \right)$, hence $\frac{3+\sqrt{3}}{2}$ is an accumulation point. Using the fact that $\Lambda$ is closed under the action of $T$, we see that there are infinitely many accumulation points in $\Lambda$.
Since the only limits of $T$-orbits are $\frac{3\pm\sqrt{5}}{2}$, it must contain infinitely many orbits in order to have infinitely many accumulation points.
\end{proof}
\begin{acknowledgments}
We thank Andre Reznikov for valuable discussions. The authors acknowledge the support ISF grants 357/13, 1017/12.
\end{acknowledgments}


\end{document}